\PassOptionsToPackage{usenames, dvipsnames}{color}

\documentclass[10pt,letterpaper,english,reqno]{amsart}

\usepackage{amsmath,amsfonts,amssymb,mathrsfs,stackrel}
\usepackage[english]{babel}
\usepackage[utf8]{inputenc}
\usepackage[T1]{fontenc}
\usepackage{graphicx}
\usepackage{tabularx}
\usepackage{mathtools}
\usepackage{hyperref}
\usepackage{cleveref}
\usepackage{comment}
\usepackage{float}
\usepackage[all]{xy}
\usepackage{xurl}
\numberwithin{equation}{subsection}

\hypersetup{colorlinks,citecolor=blue,linktocpage,hyperindex=true,backref=true}

\newcommand\cC{{\mathcal C}}

\newcommand\cE{{\mathcal E}}

\newcommand\cH{{\mathcal H}}

\newcommand\cM{{\mathcal M}}

\newcommand\cO{{\mathcal O}}

\newcommand\cS{{\mathcal S}}

\newcommand\cW{{\mathcal W}}

\newcommand\cY{{\mathcal Y}}
\newcommand\cZ{{\mathcal Z}}

\newcommand\fB{{\mathfrak B}}
\newcommand\fC{{\mathfrak C}}
\newcommand\fD{{\mathfrak D}}

\newcommand\fL{{\mathfrak L}}

\newcommand\fZ{{\mathfrak Z}}

\newcommand\C{{\mathbb C}}
\newcommand\bD{{\mathbb D}}

\newcommand\bH{{\mathbb H}}

\newcommand\N{{\mathbb N}}

\newcommand\bP{{\mathbb P}}
\newcommand\Q{{\mathbb Q}}
\newcommand\R{{\mathbb R}}

\newcommand\bV{{\mathbb V}}

\newcommand\Z{{\mathbb Z}}

\newcommand\hC{{\widehat C}}
\newcommand\hD{{\widehat D}}

\newcommand\hM{{\widehat M}}

\newcommand\wY{{\widetilde Y}}
\newcommand\wZ{{\widetilde Z}}

\newcommand\oL{{\overline L}}

\newcommand\oO{{\overline O}}

\newcommand\bfG{{\mathbf G}}
\newcommand\bfH{{\mathbf H}}

\DeclareMathOperator{\diag}{diag}

\DeclareMathOperator{\GL}{GL}

\DeclareMathOperator{\Hom}{Hom}

\DeclareMathOperator{\PGL}{PGL}

\DeclareMathOperator{\Sp}{Sp}

\DeclareMathOperator{\vol}{vol}

\usepackage{amssymb,tikz}
\newcommand\onto{\twoheadrightarrow}
\newcommand\into{\hookrightarrow}

\newtheorem{theorem}{Theorem}[section]

\newtheorem{conjecture}[theorem]{Conjecture}
\newtheorem{corollary}[theorem]{Corollary}
\newtheorem{lemma}[theorem]{Lemma}
\newtheorem{metatheorem}[theorem]{Meta-Theorem}
\newtheorem{proposition}[theorem]{Proposition}

\theoremstyle{definition}

\newtheorem{definition}[theorem]{Definition}
\newtheorem{example}[theorem]{Example}

\newtheorem{metadefinition}[theorem]{Meta-Definition}

\newtheorem{remark}[theorem]{Remark}

\usepackage[usenames,dvipsnames]{color}

\title[Non-abelian Hodge locus]{On the non-abelian Hodge locus I}%locus for cocompact lattices}
\author{Philip Engel}
\address{Department of Mathematics, University of Georgia, Boyd Hall, Athens, GA 30602}
\email{philip.engel@uga.edu}

\author{Salim Tayou}
\address{Department of Mathematics, Dartmouth College, Kemeny Hall, Hanover, NH 03755, USA}
\email{salim.tayou@dartmouth.edu}
\date{\today}

\begin{document}

\begin{abstract}

We partially resolve conjectures of Deligne and Simpson 
concerning $\Z$-local systems on quasi-projective varieties
that underlie a polarized variation of Hodge structure.
For local systems with $\Q$-anisotropic monodromy, we prove
(1) a relative form of Deligne's finiteness theorem, for any 
family of quasi-projective varieties, and (2) algebraicity 
of the corresponding non-abelian Hodge locus.
\end{abstract}

\maketitle
\setcounter{tocdepth}{1}
\tableofcontents

\section{Introduction}

Let $\Pi=\pi_1(Y,*)$ be the fundamental group of a smooth
quasi-projective variety. A fundamental result of Deligne \cite{deligne87} is that, up to conjugacy, only finitely many representations 
$\rho\colon \Pi\to \GL_n(\Z)$ underlie a $\Z$-polarized pure variation of Hodge structure ($\Z$-PVHS) over $Y$.

In this paper, we are primarily concerned with two questions:
\begin{enumerate}
\item[(Q1)] If one deforms $Y$ in a topologically trivial family $\cY\to \cS$ of smooth quasi-projective varieties, then do only finitely many representations of $\Pi$ underlie a $\Z$-PVHS on $Y_s$ for some $s \in \cS$? \vspace{2pt}
\item[(Q2)] In the relative moduli space $M_{\rm dR}(\cY/\cS,\GL_n)$ of vector bundles with flat connection, is the locus underlying a $\Z$-PVHS algebraic?
\end{enumerate}

The first question is due to Deligne \cite[Question 3.13]{deligne87}. Simpson \cite[Conjecture 12.3]{simpson} posed
and made progress on the second question, proving that this locus is analytic.

Note that the two questions are related: Q2 implies
Q1 because an algebraic set will have only finitely many 
connected components, and the representation of $\Pi$ is 
locally constant along a locus of flat connections underlying a $\Z$-PVHS.

We answer both questions, under the following assumption:

\begin{definition} 
Let $\rho\colon \Pi\to \GL_n(\Z)$ be a group representation and let $\bfH$ denote
the $\Q$-Zariski closure of ${\rm im}(\rho)$ in $\GL_n(\Q)$. We say that $\rho$ has {\it $\Q$-anisotropic monodromy} if $\bfH$ is anisotropic as an algebraic group over $\Q$, i.e.~any non-constant cocharacter $\mathbb{G}_m\to \bfH$ is central.

\end{definition}

When $\bfH$ is semisimple, as is the case for any $\Z$-PVHS, this condition is, by \cite[Thm.~11.8]{Borel-Harish-Chandra}, equivalent to $\bfH(\Z)\backslash \bfH(\R)$ being compact, where $\bfH(\Z):=\bfH(\R)\cap \GL_n(\Z)$.

\begin{theorem}\label{main} 
Let $\cY\to \cS$ be a topologically trivial family of smooth quasi-projective varieties. Then the flat connections in $M_{\rm dR}(\cY/\cS,\GL_n)$ underlying a $\Z$-PVHS with $\Q$-anisotropic monodromy form an algebraic subvariety. \vspace{2pt}

In particular, if $\Pi = \pi_1(Y_0,*)$ for some
$0\in \cS$, then only finitely many
representations of $\Pi$ underlie a $\Z$-PVHS with $\Q$-anisotropic monodromy on some fiber 
$Y_s$, up to the mapping class group action of $\pi_1(\cS,0)$.
\end{theorem}

We refer to \Cref{identification} for more details on the mapping class group action of $\pi_1(\cS,0)$ mentioned in \Cref{main}. 

A useful feature when the monodromy is 
$\Q$-anisotropic is that, due to Griffiths' 
generalization of the Borel extension theorem,
a $\Z$-PVHS on $Y_s$ extends, after a finite \'etale base change
of degree bounded solely in 
terms of $n$ and $\pi_1(Y_s)$,
over a projective simple 
normal crossings compactification $\overline{Y}_s$. This holds 
because there is
an \'etale cover of bounded degree 
$\widetilde{Y}_s\to Y_s$ for which
the pullback of any $\Z$-local system of
rank $n$ has monodromy contained
in a torsion free subgroup of $\GL_n(\Z)$.

Replacing $\cS$ with
a finite \'etale cover, we uniformly
pass to such an \'etale base change 
$\wY_s\to Y_s$ 
for all $s\in \cS$.
Then, we stratify $\cS$ 
into loci over which $\cY$
admits a relative simple normal crossings compactification. 
This is achieved by induction on dimension,
applying resolution of singularities over the 
generic point of each stratum. 
Observe
that Q1 and Q2 are Zariski-local on $\cS$.
So both Q1 and Q2 (when the monodromy is $\Q$-anisotropic)
reduce to families of smooth projective varieties. Note that the algebraicity
on a finite \'etale cover of $\cS$ implies
it for $\cS$ itself. Hence,
for the remainder of the paper, we will assume that
$\cY\to \cS$ is smooth projective, and
$\cS$ is quasiprojective.

Our result also answers a question asked by Landesman and Litt \cite[Question 8.2.1]{landesman-litt}, when the monodromy is $\Q$-anisotropic. 

\subsection{The non-abelian Hodge locus}
In a seminal paper \cite{simpson-moduli-2}, Simpson defined 
$M_{\rm Dol}(\cY/\cS,\mathrm{GL}_n)$, resp.~$M_{\rm dR}(\cY/\cS,\mathrm{GL}_n)$, the relative Dolbeault space,
resp.~the relative de Rham space: 
$M_{\rm Dol}(\cY/\cS,\mathrm{GL}_n)$ is a relative moduli space
of semistable Higgs bundles $(\cE,\phi)$ with vanishing rational Chern classes and $M_{\rm dR}(\cY/\cS,\mathrm{GL}_n)$ is a relative moduli space of vector bundles with flat connection.

Let $N_{\rm Dol} \subset M_{\rm Dol}(\cY/\cS,\mathrm{GL}_n)$ be the fixed point set of the $\mathbb G_m$-action $(\cE,\phi)\mapsto (\cE,t\phi)$ and let $N_{\rm dR}$ be its image in $M_{\rm dR}(\cY/\cS,\mathrm{GL}_n)$ under the non-abelian Hodge correspondence.
Define $$M_{\rm dR}(\cY/\cS,\mathrm{GL}_n(\Z))\subset M_{\rm dR}(\cY/\cS,\mathrm{GL}_n)$$ to be the flat bundles having integral monodromy representations on a fiber of $\cY\rightarrow \cS$. Following Simpson \cite[\S 12]{simpson}, we define the non-abelian Hodge locus, called the Noether-Lefschetz locus in {\it loc. cit.}, 
 \[\mathrm{NHL}(\cY/\cS,\mathrm{GL}_n):=N_{\rm dR}\cap M_{\rm dR}(\cY/\cS,\mathrm{GL}_n(\Z)).\]  
 These are the flat vector bundles underlying a $\Z$-PVHS. It follows from Simpson's work, see \cite[Theorem 12.1]{simpson}, that the morphism $\mathrm{NHL}(\cY/\cS,\mathrm{GL}_n)\rightarrow S$ is proper, $\mathrm{NHL}(\cY/\cS,\mathrm{GL}_n)$ has the structure of a complex analytic space, and that both inclusions $\mathrm{NHL}(\cY/\cS,\mathrm{GL}_n)\hookrightarrow M_{\rm dR}(\cY/\mathcal S,\GL_n)$ and $\mathrm{NHL}(\cY/\cS,\mathrm{GL}_n)\hookrightarrow M_{\rm Dol}(\cY/\cS,\GL_n)$ are complex analytic. 
 
 As a consequence of the non-abelian Hodge conjecture, see \cite[Conjecture 12.4]{simpson}, Simpson makes the following prediction, see \cite[Conjecture 12.3]{simpson}. 
\begin{conjecture}\label{conjecture}
The analytic variety $\mathrm{NHL}(\cY/\cS,\mathrm{GL}_n)$ is an algebraic variety and the inclusions into $M_{\rm dR}(\cY/\mathcal S,\GL_n)$ and $M_{\rm Dol}(\cY/\cS,\GL_n)$ are algebraic morphisms.  
\end{conjecture}

When the base $\cS$ is projective, Conjecture \ref{conjecture}
follows from Serre's GAGA theorem \cite{gaga}, see
\cite[Corollary 12.2]{simpson}.
Conjecture \ref{conjecture}
is the non-abelian analogue
of the main theorem of 
Cattani--Deligne--Kaplan \cite{CDK},
that the locus of Hodge
classes is algebraic, which is a 
consequence of the classical Hodge
conjecture.

There is a decomposition
$$ \mathrm{NHL}(\cY/\cS,\mathrm{GL}_n) =
\mathrm{NHL}_a(\cY/\cS,\mathrm{GL}_n)\sqcup \mathrm{NHL}_{i}(\cY/\cS,\mathrm{GL}_n)$$ according to whether the monodromy representation
is $\Q$-anisotropic or $\Q$-isotropic. Our main 
Theorem \ref{main} proves \Cref{conjecture} for the locus $\mathrm{NHL}_a(\cY/\cS,\mathrm{GL}_n)$. The case of
$\Q$-isotropic monodromy will be explored in future work.

\subsection{Strategy of the proof}
The proof splits into two parts, each of a rather different nature. We first prove Q1 using techniques from hyperbolic and metric geometry.
Then, the resolution of Q1 is used as input to prove Q2, using more algebraic and analytic techniques.

\subsubsection{Finiteness of monodromy representations}
By slicing by hyperplanes, Q1 can be reduced to the case of curves,
and in turn, to the universal family $\cC_{g,n}\to \cM_{g,n}$ of curves of genus $g\geq 2$ and with $n$ punctures, $n\geq 0$. Our assumption of the monodromy being $\Q$-anisotropic allows us to reduce to the case $n=0$. Let $$\Phi\colon C\to \Gamma\backslash \bD$$
be the period map of a $\Z$-PVHS with $\Q$-anisotropic monodromy on some $C\in \cM_g$.
Every genus $g$ Riemann surface $C$ admits a hyperbolic metric, and
Deligne's finiteness result relies critically
on the length-contracting property of $\Phi$
\cite[10.1]{griffithsIII}. But as
the curve $C\in \cM_g$ degenerates, the length-contracting property
alone ceases to be useful: The monodromy representation
will be determined by curves whose hyperbolic geodesic
representatives have length growing to infinity.

These geodesics grow in length as they cross hyperbolic collars
forming near the nodes of the limiting curve.
Thus, our key lemma (\Cref{collar-bound}) is that the image of a length-decreasing
harmonic map from a hyperbolic collar to a symmetric space is 
bounded, even as the transverse length to the collar grows to 
infinity.

\subsubsection{Algebraicity of ${\rm NHL}_a(\cY/\cS,\GL_n)$}

Our main tool for proving Q2 is an algebraization theorem
for Douady spaces of compact analytic subspaces
of Hodge manifolds $\Gamma\backslash \bD$ that are tangent to the Griffiths distribution and which
parameterize period images of $\Z$-PVHS's
with big monodromy.

The local analytic branches of the non-abelian Hodge
locus are the isomonodromic deformations of a fixed integral
representation which underlie a $\Z$-PVHS.
The fibers of $\cY\rightarrow \cS$ along a branch
admit a period map $\Phi_s\colon Y_s\to \Gamma\backslash \bD$.
The images $\Phi_s(Y_s)$ of such period maps are 
closed analytic spaces,
tangent to the Griffiths distribution on $\Gamma\backslash \bD$,
of bounded volume with respect to the Griffiths line bundle.

When $\Gamma\backslash \bD$ is compact, we prove that such period
images are parameterized by a product of a compact Moishezon space
and a sub-period domain of $\bD$ accounting for the factors where the
monodromy representation is finite. We identify the non-abelian Hodge
locus as a relative space of maps of bounded degree from $\cY/\cS$ to the universal family over the Moishezon space.

Then Q2 follows for period maps with a fixed target $\Gamma\backslash \bD$.
The set of such arithmetic quotients $\Gamma\backslash \bD$ which
can appear is bounded using the resolution of Q1. Theorem \ref{main} follows.

\subsection{Organization of the paper}
In \S 2 we recall some background results on polarized pure variations of Hodge structures and period domains. 
In \S 3, we prove
the relative version of Deligne's finiteness theorem, for
representations with $\Q$-anisotropic monodromy.
Then in \S 4, we introduce
the Douady and Barlet spaces in the general context of polarized
distribution manifolds and prove their key properties. In \S 5,
we prove algebraicity of the $\Q$-anisotropic non-abelian Hodge locus. 

\subsection{Acknowledgements}
The first author thanks P. Smillie for suggesting a
proof of Proposition \ref{key-lemma}, R. Krishnamoorthy for many helpful discussions, and B. Bakker, B. Klingler, and D. Litt for their
insights.
The second author thanks A. Landesman for bringing this question
to his attention and for useful conversations, and D. Maulik, Y.-T. Siu, and N. Tholozan for useful conversations. We thank also the anonymous referee for helpful comments and suggestions.

The first author was supported by NSF grant DMS-2201221. The second author was supported by NSF grant DMS-2302388.

\section{Variations of Hodge structures}\label{hodge-theory}

We recall in this section some background results on 
polarized variations of pure Hodge structures and we fix 
notations. All variations of Hodge structures in this paper are pure and our main references are 
\cite{green-griffiths-kerr,klingler-survey}, see also \cite{griffithsI,griffithsIII}. 

\subsection{Monodromy and Mumford-Tate group}
Let $Y$ be a complex manifold and let 
$\bV:=(V_\Z,F^\bullet,\psi)$ be a polarized variation 
of pure Hodge structure of weight $k$ on $Y$. Here $V_\Z$
is the $\Z$-local system, $F^\bullet$ is the Hodge filtration
on $V_\Z\otimes \cO_Y$, and $\psi$ is the polarization.
Let $\bfG$ be 
the \emph{generic Mumford-Tate} group of the variation 
and let $\bfH$ be the algebraic monodromy group of $\bV$. 

We recall that $\bfG$ is the Mumford--Tate group of the Hodge
structure over a very general point of $Y$ and 
$\bfH$ is defined as follows: 
fix a base point $*\in Y$ and denote the monodromy 
representation associated to the local system $V_\Z$ by 
$\rho\colon \pi_1(Y,*)\to \GL(V_{\Z,*}),$ which lands in 
the subgroup $\Sp(V_{\Z,*})$ or ${\rm O}(V_{\Z,*})$ 
depending on the parity of the weight. Then $\bfH$ is the 
identity component of the $\Q$-Zariski closure of the image 
of $\rho$. The groups $\bfG$ and $\bfH$ are reductive 
algebraic groups over $\Q$ and by a classical theorem of 
Deligne \cite[Section 4]{hodge2} and Andr\'e  \cite[Theorem 1]{Andre-Mumford-Tate}, $\bfH$ is a normal subgroup of 
$\bfG^{\rm der}$, the derived group of $\bfG$. It follows 
that we have a decomposition over $\Q$ of the adjoint groups 
$\bfG^{\rm ad}=\bfH^{\rm ad}\times \bfH'$.    

Let $\bD$ be the \emph{Mumford-Tate domain} associated to the 
variation. It is a complex analytic space, homogeneous for 
$G:=\bfG^{\rm ad}(\R)^+$ and it can be identified with a 
quotient $G/U$ where $U\subset G$ is a compact 
subgroup. 

In terms of Hodge structures, the variation of Hodge structure $\bV$ induces, by restriction to a point $s\in S$, a pure Hodge structure. Therefore we have a decomposition $\bV_{\Z,s}\otimes_{\Z} \C=\bigoplus_{p+q=k}V^{p,q}_s$, where $\overline{V^{p,q}_s}=V^{q,p}_s$. Then $U$ is the real subgroup preserving each $V^{p,q}_*$ and the Hodge 
pairing between $V^{p,q}_*$ and $V^{q,p}_*$.
From the theory of symmetric spaces, $\bD$ is an analytic 
open subset of the {\it compact dual} $\bD^\vee$, a 
projective subvariety of a symplectic or an orthogonal flag 
variety with specified Mumford-Tate group. There exists then a parabolic subgroup $P\subset G_\C$ such that
$\bD^\vee = G_\C/P$ and $P\cap G = U$.
\medskip

The variation of Hodge structure $\bV$ on $Y$ is 
completely described by its holomorphic \emph{period map}: 
\[\Phi:Y\rightarrow \Gamma\backslash \bD,\]
where $\Gamma\subset \bfG(\Z)$ is a finite index subgroup 
preserving $V_\Z$ such that the monodromy representation 
factors through $\Gamma$. Up to taking a finite \'etale cover 
of $Y$, we can assume that $\Gamma$ is neat, hence acting 
freely on $\bD$. Then the quotient $X_\Gamma:=\Gamma\backslash \bD$ is a 
connected complex manifold, called a \emph{connected Hodge manifold}, see 
\cite[Definiton 3.18]{klingler-survey}. It is the 
classifying space of polarized $\Z$-Hodge
structures on  $V_\Z$ whose generic Mumford-Tate group is 
contained in $\bfG$, with level structure corresponding
to $\Gamma$.

In general, $X_\Gamma$ does not admit the structure
of an algebraic variety unless 
$\bD$ fibers holomorphically or anti-holomorphically over a Hermitian symmetric domain \cite{griffiths-robles}. In that case,
$X_\Gamma$ is in fact quasiprojective by 
the Baily-Borel theorem \cite{bailyborel}, and $\Phi$ is 
algebraic by the Borel hyperbolicity theorem 
\cite{borelmetric}, see also 
\cite{bakker-brunebarbe-tsimerman} for another proof.

We can furthermore refine the period map by taking into 
account the algebraic monodromy group $\bfH$. The Mumford--Tate domain $\bD$ decomposes according to the 
decomposition 
$\bfG^{\rm ad}=\bfH^{\rm ad}\times \bfH'$ of adjoint groups as 
$\bD=\bD_H\times \bD_{H'}$ where $\bD_H$ is an 
$H:=\bfH^{\rm ad}(\R)^{+}$-homogeneous space. Up to a finite 
\'etale cover of $Y$, we can assume that the lattice $\Gamma$ 
decomposes as $\Gamma=\Gamma_{H}\times \Gamma_{H'}$ 
where $\Gamma_{H}\subset \bfH(\Z)$ and 
$\Gamma_{H'}\subset \bfH'(\Z)$ are arithmetic subgroups. 
Then the projection of the period map $\Phi$ is constant on 
the second factor and hence the period map takes the 
following shape: 
\[\Phi:S\rightarrow \Gamma_H\backslash \bD_H\times 
\{t_Y\}\hookrightarrow \Gamma\backslash \bD,\]
where $t_Y$ is a Hodge generic point in $\bD_{H'}$. 
So $X_{\Gamma_H}\times \bD_{H'}$ serves as a
classifying space of $\Z$-PVHS on a lattice isometric to $V_{\Z,*}$ whose generic Mumford-Tate group is contained in $\bfG$, and whose monodromy factors through $\Gamma_H$. The classifying map for
such a variation factors through the inclusion of $X_{\Gamma_H}\times \{t\}$ for some fixed $t$.

\subsection{Automorphic vector bundles}
We refer to \cite[Section 12.1]{carlson} for more details on this section. Given any complex linear representation of
$\chi\colon U\rightarrow \mathrm{GL}(W)$, there is an associated 
holomorphic vector bundle $G\times_U W\rightarrow \bD$
which is $\Gamma$-equivariant and hence descends to a holomorphic vector bundle over $X_\Gamma$. In particular, for any $p$, the natural representation of $U$ on $V^{p,q}_*$ defines a holomorphic vector bundle on $\bD$ which is identified to the $p{\rm th}$ graded piece $F^p/F^{p+1}$ of the Hodge filtration. 

Any character $\chi\colon U\to \mathbb{S}^1$ defines an equivariant 
holomorphic line bundle $L_\chi\rightarrow \bD$. For example,
if the character $\chi$ is the determinant of the action of
$U$ on $V^{p,q}_*$, we get the line bundle
$L_p=\det(F^p/F^{p+1}).$ Any such equivariant line
bundle admits a unique (up to scaling) left
$G$-invariant hermitian metric 
$$h\colon L_\chi\otimes \oL_{\chi}\to \C.$$

\begin{definition} The {\it Griffiths bundle}
$L\to X_\Gamma$ is defined by 
$$\textstyle L:=\bigotimes_{p\geq 0} (L_p)^{\otimes p}.$$ \end{definition}

We denote the descent
to $X_\Gamma$ of the equivariant vector bundles $F^p$, line bundles 
$L_p$, and the hermitian metrics $h$ by the same symbols.

\begin{remark} While $F^\bullet$ defines a filtration 
of holomorphic vector bundles over $X_\Gamma$, it does 
not, in general, define a $\Z$-PVHS over $X_\Gamma$
for the tautological local system, as Griffiths transversality condition fails. \end{remark}

Recall that the tangent space to the Grassmannian 
at a subspace $W\subset V$ is canonically isomorphic
to $\Hom(W,V/W)$. Since $\bD$ is an open subset of a
flag variety $\bD^\vee$, we have an inclusion
$$\textstyle T\bD\subset \bigoplus_p \Hom(F^p, V/F^p).$$

The Griffiths transversality condition on a $\Z$-PVHS
over $Y$ implies that the differential $d\Phi$ of the period
map lands in an appropriate subspace of the tangent space:

\begin{definition} The {\it Griffiths horizontal distribution}
$\Xi\subset T\bD$ is the  holomophic subbundle of the tangent
bundle defined by $$\Xi_{F^\bullet}:=T_{F^\bullet}\bD \cap 
\textstyle \bigoplus_p \Hom(F^p,F^{p-1}/F^p).$$ It is 
$G$-invariant, and so descends to a distribution in 
$TX_\Gamma$ which we also denote by $\Xi$. \end{definition}

The following proposition is 
\cite[Prop.~7.15]{griffithsIII}.

\begin{proposition} Let $\omega_L:=\frac{i}{2\pi}\partial 
\overline{\partial} \log h\in \Lambda^{1,1}(X_\Gamma, \R)$ 
be the curvature form of the Hermitian metric $h$ on $L$. 
Then $\omega_L\big{|}_{\Xi}$ is positive definite, in the sense 
that for any nonzero $v\in \Xi_\R$, 
$$\omega_L(v,Jv)>0.$$\end{proposition}

From this, Griffiths concluded that the image 
of $\Phi$ admits a holomorphic line bundle with positive 
curvature. In particular, using a generalization of the 
Kodaira embedding theorem due to Grauert, he proved, see 
\cite[Thm.~9.7]{griffithsIII}:

\begin{theorem} Let $\Phi\colon Y\to X_\Gamma$ be the 
period map of a $\Z$-PVHS on a compact, complex manifold $Y$. 
Then $\Phi(Y)$, with its reduced analytic space structure,
is a projective algebraic variety.\end{theorem}

 It seems though, that some conditions of Grauert's
 theorem do not always hold.
 In particular, it may not be the case that we have an inclusion of Zariski tangent spaces
 $T\Phi(Y)\subset \Xi$ due to singularities on $\Phi(Y)$.
 An independent proof and strengthening to the 
 non-compact case was given in
 \cite[Thm.~1.1]{bakker-brunebarbe-tsimerman}.

\section{Boundedness of monodromy representations}

Let $\cS$ be a smooth connected quasi-projective complex algebraic variety and let $\pi:\cY\to \cS$ be a smooth projective morphism. Our goal in this section is to prove that there are only finitely representations $\pi_1(Y_0)\rightarrow 
\GL_n(\Z)$, up to conjugacy,
which underlie a $\Z$-PVHS with $\Q$-anisotropic monodromy on some fiber $Y_s$ of $\pi:\cY\rightarrow\cS$, after an identification $\pi_1(Y_0,*)\simeq\pi_1(Y_s,*)$ moving the base point
in the universal family.
%many groups $\Gamma$, up to $\Q$-conjugacy, which can appear as the $\Z$-Zariski 

%closure of the monodromy of a $\Z$-PVHS of bounded rank $n$,
%on a fiber of $\cY$. 

Slicing $\cY$ by hyperplanes,
we can apply the Lefschetz theorem to reduce to the case
of a relative smooth projective curve $\cC\to \cS$ (passing to a finite Zariski
cover of $\cS$ if necessary). Then, we may
as well assume that $\cS=\cM_g$ and that $\cC=\cC_g$
is the universal curve. This is a particular instance of a question asked by Deligne, for representations with $\Q$-anisotropic monodromy, see \cite[Question 3.13]{deligne87}.
\medskip 

We can decompose $\cM_g$ into two subsets, the {\it thick}
part and the {\it thin} part. Let $C\in \cM_g$ be a Riemmann surface of genus $g$ and let $\gamma\in \pi_1(C)$ be a loop. Then $C$ has a unique hyperbolic metric of constant curvature $-1$, 
in the conformal equivalence class defined
by the complex structure on $C$. There is a unique
representative of the free homotopy class of $\gamma$
which is a hyperbolic geodesic for this metric. Let $\ell_C(\gamma)$ denote its hyperbolic length. Then, the thick part of $\cM_g$
is a compact subset
$\cM_g^{\geq \epsilon}\subset \cM_g$ consisting
of all curves $C\in \mathcal{M}_g$, for which $\ell_C(\gamma)\geq \epsilon$
for all $\gamma\in \pi_1(C)$, see \cite[Cor.~3]{mumford-thick}.

First, we deal with the thick part. The proof follows,
nearly verbatim, Deligne's proof \cite{deligne87}
of finiteness of monodromy representations underlying
$\Z$-PVHS on a fixed curve $C$. 

\begin{definition}\label{identification}
  Let $\Pi_g$ be the surface group:
  $$\Pi_g = \langle \alpha_1,\beta_1,\dots,\alpha_g,\beta_g \,\big{|}\,\textstyle \prod_{i=1}^g\alpha_i\beta_i\alpha_i^{-1}\beta_i^{-1}=1\rangle.$$
  Fix a pointed Riemann surface $(C_0,*_0)\in \cM_{g,1}=\cC_g$ of genus $g$ and an isomorphism $\pi_1(C_0,*_0)\simeq \Pi_g$.
Then a path in $\cC_g$ connecting $(C_0,*_0)$ to $(C,*)$
produces an identification $$\pi_1(C,*)\simeq \pi_1(C_0,*_0)\simeq \Pi_g.$$ We call such an
identification \emph{admissible}.
  \end{definition}

 Two such admissible
identifications can be compared by an 
automorphism of $\Pi_g$ induced by a path
from $(C_0,*_0)$ to itself, i.e., an element
of $\pi_1(\cC_g, (C_0,*_0))$.
The paths connecting $(C_0,*_0)$ to itself
keeping $C_0\in \cM_g$ constant in moduli
induce the inner automorphisms ${\rm Inn}(\Pi_g)$. 
The paths connecting  $(C_0,*_0)$ to itself by moving 
$C_0\in \cM_g$ in moduli induce an inclusion of
the mapping class group ${\rm Mod}_g\subset {\rm Out}(\Pi_g)$
as an index $2$ subgroup of the outer automorphism group ${\rm Out}(\Pi_g)$, corresponding
to orientation, see \cite[Theorem 8.1]{farb-margalit}. So any isomorphism 
$\pi_1(C,*)\simeq \Pi_g$ induced by an oriented
homeomorphism $(C,*)\to (C_0,*_0)$ is admissible.

\begin{proposition}\label{deligne-thick} Let 
$\rho\colon \pi_1(C,*)\to \GL_n(\Z)$ be the 
monodromy representation of a $\Z$-PVHS of rank $n$
on some $C\in \cM_g^{\geq \epsilon}$ in the thick part of
the moduli space. There is an admissible identification $\pi_1(C,*)\simeq \Pi_g$ identifying $\rho$ with
one of a finite list of representations $\Pi_g\to \mathrm{GL}_n(\Z)$, up to conjugacy. \end{proposition}

\begin{proof}
A theorem of Procesi \cite{procesi}
states that, up to conjugacy, a semisimple representation
$\rho\colon \Pi\to \GL_n(\C)$
from any finitely generated group $\Pi$
is uniquely determined by the function
\begin{align*} \{1,\dots,m\}&\to \C \\ 
j&\mapsto {\rm tr}(\rho(\delta_j))
\end{align*}
for some finite generating set 
$(\delta_j)_{1\leq j\leq m}$
of the group, where $m$ depends only on $\Pi$ 
and $n$. 

Choose, for once and all, such a generating
set $\delta_1,\dots,\delta_m$ for the surface group
$\Pi_g$. We call this set the
{\it Procesi generators}.
Deligne's argument relies on the famous length-contracting property of period maps,
due to Griffiths \cite[10.1]{griffithsIII}:

\begin{theorem}\label{length-contracting} There is a $G$-invariant metric on $\bD=G/U$ for which any holomorphic, 
Griffiths transverse map $\Delta\to \bD$ from a holomorphic disk 
is length-contracting for the hyperbolic metric on $\Delta$.
\end{theorem}

Choose a cover of $\cM_g^{\geq \epsilon}$ by a finite number
of contractible, compact subsets $\{V_i\}_{i\in I}$.
Choosing a base-point consistently over $V_i$,
the fundamental groups $\pi_1(C,*)$
for all $C\in V_i$ are uniquely identified, by the
contractibility of $V_i$. Let $\pi_1(C,*)\simeq \Pi_g$
be an admissible identification, and consider
the resulting family of Procesi generators $(\delta_j)_{1\leq j\leq m}$
of $\pi_1(C,*)$ for $C\in V_i$. Then $\ell_C(\delta_j)$
is a continuous function on $V_i$ which, by compactness,
is bounded. Hence there exists some $M$ for which
$\ell_C(\delta_j)\leq M$ for all $1\leq j\leq m$ and all $C\in V_i$.

Suppose that $\rho\colon \pi_1(C,*)\to \Gamma$ is the monodromy
representation of a $\Z$-PVHS for some $C\in V_i$.
Then, applying Theorem \ref{length-contracting} to the
hyperbolic uniformization $\Delta\to C$,
we conclude that there exists a point $x\in \bD$
for which $d_{\bD}(x,\rho(\delta_j)\cdot x)\leq M$. In particular,
$x$ may be taken as the period image of some
point on the lift to $\Delta$ of the hyperbolic geodesic representing $\delta_j$.
Thus, $\rho(\delta_j)$ has bounded translation length,
and thus, bounded trace, by Lemma
\ref{translation-bound}. See \cite[Corollaire 1.9]{deligne87}.

\begin{lemma}\label{translation-bound}
Let $g\in G$ and suppose that $d_\bD(x,g\cdot x)\leq M$
for some $x\in \bD$. Then $|{\rm tr}(g)|\leq N$, for some $N$ depending only on $\bD$ and $M$. 
\end{lemma}

\begin{proof} Fix a base point $x_0\in \bD$ and choose some $h\in G$
for which $h\cdot x_0 = x$. Then
$$d_\bD(x,g\cdot x) = d_\bD(h\cdot x_0, gh\cdot x_0) = 
d_\bD(x_0,h^{-1}gh\cdot x_0)\leq M.$$
Since the closed ball of radius $M$ around $x_0$ is compact,
and the map $G\to G/U=\bD$ has compact fibers,
we conclude that the set $$\{k\in G\,\big{|}\,
d_\bD(x_0,k\cdot x_0)\leq M\}$$ is compact. As the trace
is a continuous function, we conclude that ${\rm tr}$ is bounded
on the above set, in terms of $M$ alone. We conclude that
${\rm tr}(h^{-1}gh)={\rm tr}(g)$ is bounded.
\end{proof}

Hence the trace ${\rm tr}(\rho(\delta_j))$ is
bounded in terms of $\ell_C(\delta_j)\leq M$,
and hence it is bounded globally
on $V_i$ by some integer $N$. It is furthermore an integer,
as $\rho$ lands in $\GL_n(\Z)$. Since there are only
finitely many possibilities for a map
$\{1,\dots,m\}\to \{-N,\dots,N\}$,
there are only finitely many monodromy representations
achieved for a $\Z$-PVHS over any $C\in V_i$.
Since the indexing set $I$ is finite,
we conclude the same over $\cM_g^{\geq \epsilon}$, up to conjugacy.
\end{proof}

Thus, it remains to consider the thin part of the moduli
space $\cM_g^{<\epsilon}$ consisting of smooth curves with systole less than $\epsilon$.

\begin{definition}\label{collar-def}
A {\it collar} $A$ is the
Riemann surface with boundary 
$$\left.\left\{
re^{i\theta}\in \bH
\left|
\begin{array}{l}
			1\leq r \leq r_0 \\
			\theta_0\leq \theta \leq \pi-\theta_0
\end{array}
\right.\!\!
\right\}\right/\sim$$
where $\tau\sim r_0\tau$.
A {\it half-collar}
is the subregion where $\theta\leq \tfrac{\pi}{2}$. 
\end{definition}

A collar admits a Riemannian metric of constant curvature $-1$ induced by the Poincar\'e metric on $\mathbb{H}$. We recall a famous result
due to Keen \cite{keen}. The sharpness is due
to Buser \cite[Thm.~C]{buser}.

\begin{lemma}[Collar Lemma]\label{universal-collar}
Every simple closed geodesic $\gamma$
of length $\ell$ on a complete
hyperbolic surface $C$ is contained in a
hyperbolic collar $A_\gamma\subset C$ of transverse
length $\ln \big(\frac{e^{\ell/2}+1}{e^{\ell/2}-1}\big)$.
Furthermore, any two such collars associated
to disjoint geodesics are disjoint. \end{lemma}

The function $$F(\ell):=\ln\left(\frac{e^{\ell/2}+1}{e^{\ell/2}-1}\right)$$ satisfies
$\lim_{\ell\to 0^+} F(\ell)=+\infty$, and is
monotonically decreasing towards zero as $\ell\to +\infty$.
In terms of the constants $r_0,\theta_0$ of 
Definition \ref{collar-def}, we have $r_0 = e^\ell$ and $\theta_0=\cos^{-1}(e^{-\ell/2})$.
The perimeter of
a boundary component of this collar is $\ell(1-e^{-\ell})^{-1/2}.$
More generally, the formula for the perimeter of a collar is ${\rm Per}(A)= \ell\csc(\theta_0).$

For $C\in \cM_g^{<\epsilon}$, let
$\{\gamma_1,\dots,\gamma_k\}$ be the non-empty set of simple closed curves
of hyperbolic length less than $\epsilon$. Choosing $\epsilon$
smaller than the fixed point of the function $F(\ell)$,
we conclude that all such curves are disjoint. 
So $k\leq 3g-3$, with equality when
$\{\gamma_1,\dots,\gamma_k\}$ form a pair-of-pants
decomposition of $C$.

We now recall the result of Bers \cite{bers1, bers2}:

\begin{theorem}\label{pair-of-pants} There exists a constant $B_g$
for which any hyperbolic surface of genus $g$ admits a
pair-of-pants decomposition, all of whose curves have
length bounded above by $B_g$.
\end{theorem}

By choosing $\epsilon$ so that $F(\epsilon)>B_g$,
any such pair of pants decomposition {\it must} contain
all simple closed curves of length less than $\epsilon$,
as any pair of pants decomposition not including
$\gamma_j$ would include a curve that crossed the collar of Lemma
\ref{universal-collar}.
Thus, we may extend the set $\{\gamma_1,\dots,\gamma_k\}$ to a
full pair of pants decomposition
$\{\gamma_1,\dots, \gamma_{3g-3}\}$ in such a way
that $\ell_C(\gamma_j)\leq B_g$ for all $j$.

Up to conformal equivalence, a pair of pants $P(\ell_1,\ell_2,\ell_3)$
is uniquely specified by the three cuff lengths
$\ell_1,\ell_2,\ell_3\in \R^+$. Two adjacent
pairs of pants, glued along $\gamma_i$
in a pants decomposition of $C$, contain a collar
$A_{\gamma_i}$ of transverse length
at least $F(\ell_C(\gamma_i))$, but with the bounds
$B_g$ on the chosen pairs of pants, we can
do better:

\begin{proposition}\label{def-perim} Suppose $P(\ell_1,\ell_2,\ell_3)$ is a
pair of pants with $\ell_i\leq B_g$. There
exists a constant $C_g>0$ for which each cuff is contained in
a half-collar of perimeter at least $C_g$. \end{proposition}

\begin{proof}
The key is to observe that even as $\ell_i\to 0$,
the geometry of $P(\ell_1,\ell_2,\ell_3)$ converges,
with the cuff $\gamma_i$ limiting to a hyperbolic cusp,
and the half-collars limiting to the horoball neighborhoods. Therefore $P(\ell_1,\ell_2,\ell_3)$
makes sense, for all $0\leq \ell_i\leq B_g$.
For each such surface,
each cuff (resp. cusp) has a
half-collar (resp. horoball) neighborhood
of non-zero perimeter. The supremum of
such perimeters is a continuous function
on the compact set $[0,B_g]^3$,
never equal to zero, and thus has a nonzero minimum.
\end{proof}

\begin{figure}
\includegraphics[width=2.8in]{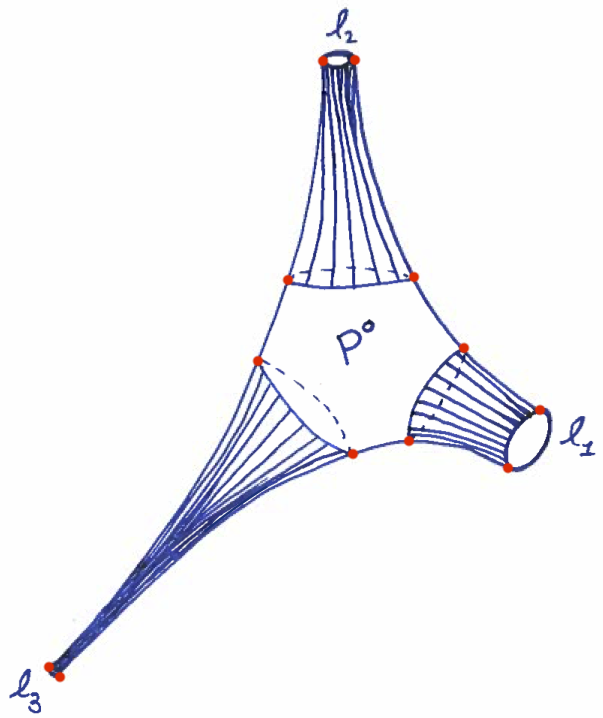}
\caption{A hyperbolic pair of pants $P(\ell_1,\ell_2,\ell_3)$,
and its truncation $P^o(\ell_1,\ell_2,\ell_3)$.
Distinguished boundary points on
$P$, $P^o$ are shown in red.}
\label{fig1}
\end{figure}

\begin{definition} The {\it truncated pair of pants} 
$P^o(\ell_1,\ell_2,\ell_3)$ (Fig.~\ref{fig1})
is the complement of the
half-collars in $P(\ell_1,\ell_2,\ell_3)$
with perimeter $C_g$. 
\end{definition}

 If $\ell_i\geq C_g$ we need not truncate the corresponding
 cuff. Making $C_g$ sufficiently small, we may assume
 that the (up to) three half-collars we cut from $P(\ell_1,\ell_2,\ell_3)$
are disjoint.

\begin{remark} The issue with truncating pairs of pants
by the universal collar of Lemma \ref{universal-collar} is
that the limit of its perimeter is
$$\lim_{\ell\to 0^+} \ell(1-e^{-\ell})^{-1/2}=0.$$ So the 
universal collar is not sufficient to bound the
geometry (e.g. as measured by the hyperbolic diameter)
of the truncated pair of pants, when $\ell\to 0$. Hence the 
need for Proposition \ref{def-perim}. \end{remark}

Consider the three seam geodesics
connecting cuffs of $P(\ell_1,\ell_2,\ell_3)$. These 
seams intersect each boundary component of
$P^o(\ell_1,\ell_2,\ell_3)$ and 
$P(\ell_1,\ell_2,\ell_3)$ at two points, see Figure \ref{fig1}. 
We call these (six total) points the {\it distinguished boundary points}
of $P^o(\ell_1,\ell_2,\ell_3)$ and
$P(\ell_1,\ell_2,\ell_3)$.
Note that the distinguished points
on a given cuff are diametrically opposite.
So when two pants are glued, the four total
distinguished points on the cuff alternate which pair of pants they come from, or the distinguished points from one pair of pants coincide with those from the other.

\begin{proposition}\label{bounded-truncated} Suppose
$\ell_1,\ell_2,\ell_3\leq B_g$ for some constant $B_g$.
Let $\mu$ be a homotopy class of paths on
the truncated pair of pants $P^o(\ell_1,\ell_2,\ell_3)$,
terminating at two distinguished points of the boundary.
Then, $\mu$ has a representative of bounded
distance $D_\mu$ independent of $\ell_i$. \end{proposition}

\begin{proof} The minimal length representative of $\mu$
on any truncated pair of pants is finite, and furthermore,
this minimal length is continuous as one varies the $\ell_i$.
This holds even when some $\ell_i=0$, corresponding to cusped
pairs of pants. The proposition follows because 
$(\ell_1,\ell_2,\ell_3)$ is restricted to lie in the compact set $[0,B_g]^3$.
\end{proof}

%Thus, the length-contracting property will be useful
%inside the truncated pair of pants. But the following 
%issue arises: As the core curve of a collar $A$ shrinks,
%the length of any transverse geodesic grows. So the 
%length-contracting property ceases to be useful on the 
%collars, at least on its own. Thus, 
The next proposition is absolutely crucial.

\begin{proposition}\label{key-lemma} Let $(M,g)$ be a simply connected
Riemannian manifold with non-positive sectional curvature and let $\Psi\colon A\to M$ be a length-contracting, harmonic map from a collar. Assume
the perimeter of $A$ is bounded above by $C_g$.
Then, the image of $A$ is contained in a ball of bounded radius 
$\tfrac{1}{2}(C_g+\pi)$. \end{proposition}

\begin{proof}
Recall that the collar $A$ is parameterized
by polar coordinates $(r,\theta)\in \bH$ (Def.~\ref{collar-def}) where $r\in \R_{>0}/(r_0)^\Z$ is the circle coordinate on the collar, and $\theta\in [\theta_0,\pi-\theta_0]$ is the transverse
coordinate. Let $p_0$ be a point on the boundary component
of $A$ defined by $\theta=\theta_0$.
Define \begin{align*} d\colon A 
&\to \R_{\geq 0} \\
q&\mapsto {\rm dist}_g(\Psi(p_0),\Psi(q)).\end{align*}

As $M$ has non-positive sectional curvature and
$\pi_1(M)$ is trivial, the distance function
${\rm dist}_g(\Psi(p_0),\cdot)\colon M\to \R_{\geq 0}$ is convex, see \cite[Corollary 4.8.2]{jost}. The 
composition of a convex function with a harmonic function
is subharmonic, so the function
$d$ is subharmonic. Let $S^1(q)$ denote the 
circle containing $q\in A$ (varying
only the coordinate $r$) and
define $$d_{\rm max}(\theta):=\max_{q'\in S^1(q)} \,d(q'),$$
which is now circularly symmetric, and so
is only a function of $\theta$.
It suffices to
prove that $d_{\rm max}$ is bounded.

Since the rotation action on $A$ is conformal,
the pullback along the rotation action of $d(q)$ is
subharmonic. Thus $d_{\rm max}(\theta)$, as a maximum of
subharmonic functions, is also subharmonic.

The hyperbolic metric is $y^{-2}(dx^2+dy^2)$
on the upper half-plane, therefore \[g_{\rm hyp}(\frac{\partial}{\partial \theta}, \frac{\partial}{\partial \theta})=1,\quad  \textrm{when}\quad  \theta = \tfrac{\pi}{2}.\]
It follows that the length-contracting property,
along with the triangle inequality, implies
\begin{align*} \left|\tfrac{\partial}{\partial\theta}(d(q))\right|\leq 1 \textrm{ when }\theta(q)=\tfrac{\pi}{2}.\end{align*}
Hence 
\begin{align*}
    \left|\tfrac{d}{d\theta}(d_{\rm max}(\theta))\right|\leq 1 \textrm{ when }\theta=\tfrac{\pi}{2}.
\end{align*}

Thus $d_{\rm rel}(\theta):=
d_{\rm max}(\theta)-\theta$ has a non-positive
derivative at $\theta=\tfrac{\pi}{2}$.
On the other hand, $\theta$ is harmonic so
$d_{\rm rel}(\theta)$ is again subharmonic.
As a subharmonic function with a non-positive derivative
at $\tfrac{\pi}{2}$, we have that
$d_{\rm rel}(\theta)$ is bounded above by its value
at the left endpoint $p_0$ for all $\theta\leq \tfrac{\pi}{2}$.
Let $D\leq \tfrac{1}{2}{\rm Per}(A)\leq \tfrac{1}{2}C_g$ denote the hyperbolic diameter of
a boundary component of $A$. By the length-contracting
property, we have $d_{\rm rel}(\theta_0)\leq D-\theta_0$ so
$$d_{\rm max}(\theta)\leq D+ (\theta-\theta_0)<\tfrac{1}{2}(C_g+ \pi)\textrm{ for all }\theta\leq \tfrac{\pi}{2}.$$

Applying the same argument to a point
$p_0$ on the other boundary component of the collar,
we conclude that for a point $p'$ on the core curve,
the ball of radius $\tfrac{1}{2}(C_g+\pi)$ about its 
image contains the image of the boundary of $A$ 
entirely. We conclude the result
by the maximum principle, as $q\mapsto {\rm dist}_g(\Psi(p'),\Psi(q))$
is subharmonic.
\end{proof}

\begin{lemma}\label{bound-away}
There is a constant $\mu_n>0$ depending only on
$n$ such that: For any arithmetic group $\Gamma$ acting on a
period domain $\bD$ classifying $\Z$-PVHS of rank at most $n$,
and for any $p\in \bD$, we have
$$d_\bD(p,\gamma(p))>\mu_n\textrm{ for all }\gamma\in \Gamma\textrm{ non-quasi-unipotent}.$$
\end{lemma}

\begin{proof} There are only finitely many possible spaces $\bD$, 
corresponding to real Lie groups $G$
of Hodge type and bounded rank,
and compact subgroups $U\subset G$, up to conjugacy. 
%Let $\chi_\gamma(t)$ denote the characteristic polynomial of $\gamma$.
Since it is monic of degree $n$, we can apply the following effective form of Kronecker's theorem due to \cite[Corollary]{blanksby}, see also the recent work of Dimitrov \cite[Theorem 1]{dimitrov-schinzel} which provides the sharpest bounds.

\begin{theorem}[\cite{blanksby}]\label{blanks} Let $\alpha$ be an algebraic integer of degree $d\leq n$.
Either $\alpha$ is a root of unity, or the largest Galois conjugate of $\alpha$ has absolute value at least $$c_n = 1+\frac{1}{30n^2\log(6n)}.$$
\end{theorem}

Factoring $\chi_\gamma(t)$ into irreducible factors, 
this theorem bounds the norm of the
largest eigenvalue of  $\gamma$ away from $1$,
whenever $\gamma$ is non-quasi-unipotent.
Let $\lambda_1, \dots, \lambda_n$ be these
eigenvalues and let 
$$L_\gamma:=\inf_{p\in \bD}
d_\bD(p,\gamma(p))$$ be the translation length.
%As $L_\gamma$ is conjugation-invariant,
%it is solely a function $F_{\bD}(\lambda_1,\dots,\lambda_n)$ 
%of the eigenvalues.

Let $S=G/K$ be the symmetric space
associated to the real group $G$. Here
$K\subset G$ is a maximal compact
subgroup containing $U$. Consider the map $$\bD=G/U\xrightarrow{\pi} G/K=S.$$ For appropriate left $G$-invariant
metrics, this map is length-contracting.
Then, $L_\gamma\geq \inf_{p\in S} d_S(p,\gamma(p))$.
We have, see \cite[Ex.~7.1]{breuillard-fujiwara},
\[d_S(p,\gamma\cdot p)=\sqrt{(\log a_1(p))^2+\cdots
+(\log a_n(p))^2},\]
where $\gamma=k_1ak_2$ with $k_1,k_2\in K_p$ and $a=\diag(a_1(p),\ldots,a_n(p))\in \R_+^n$ is the Cartan decomposition
of $\gamma$ with respect to the compact isotropy
group $K_p = {\rm Stab}_G(p)$. We have $\max_i a_i(p)\geq \max_i |\lambda_i|$ and thus we conclude $L_\gamma\geq \max_i \log |\lambda_i|$.

Hence, taking $\mu_n<\log|c_n|$ and applying
Theorem \ref{blanks},
we conclude that $L_\gamma> \mu_n$
for non-quasi-unipotent $\gamma$.
\end{proof}

\begin{corollary}\label{monodromy-trivial}
Consider a $\Z$-PVHS of rank $n$
with $\Q$-anisotropic monodromy over a curve $C$. Up to passing to
a finite \'etale cover of fixed degree, there is an $\epsilon>0$
such that, for any $\gamma\in \pi_1(C)$ with $\ell_C(\gamma)<\epsilon$, the monodromy of $\gamma$ is trivial: $\rho(\gamma)=I\in \Gamma.$
\end{corollary}

\begin{proof} This follows from Lemma \ref{bound-away},
the length-contracting property, and the 
fact that in the compact
type case, the only quasi-unipotent elements of $\Gamma$ are of finite order. Note that for all possible $\Gamma\subset \GL_n(\Z)$, the torsion can be killed at a fixed finite level, since
this holds for the entire group $\GL_n(\Z)$. \end{proof}

\begin{proposition}\label{collar-bound}
Let $(C,\gamma)$ and $\epsilon$ be as above, and let
$A$ be a hyperbolic collar on $C$ containing $\gamma$, of perimeter $C_g$.
Then the period map $A\to \Gamma\backslash \bD$ lifts to 
a period map $\Phi\colon A\to \bD$. Furthermore, the image
of $\Phi$ is contained in a ball of bounded radius $B$. 
\end{proposition}

\begin{proof} The restriction of the period map to $A$ lifts to $\bD=G/U$ by Corollary \ref{monodromy-trivial}, because the
monodromy of the core curve is trivial, and the core curve
generates $\pi_1(A)$.

Define $\Psi = \pi\circ \Phi$ to be the composition
of the period map $\Phi\colon H\to \bD$ with the quotient map $\pi\colon \bD \to S=G/K$ to the symmetric space. Then $\Psi$ is a harmonic map, as the composition of a holomorphic horizontal map and $\pi$, \cite[Theorem 1.1]{Lu}.

Applying Proposition \ref{key-lemma}, we conclude that
for $p,q$ two points on the two boundary components of $A$,
the distance $$d_S(\Psi(p), \Psi(q))$$ is bounded.
Here we use that $S$ is non-positively curved, simply
connected, and that
$\pi \colon \bD\to S$ is distance-decreasing, so 
$\pi\circ \Phi$ is also distance-decreasing.
The fibers of $\pi$ are isometric, compact submanifolds
$K/U\subset \bD$.
We conclude that the distance between
$\Phi(p)$ and $\Phi(q)$
is also bounded.
\end{proof}

We now cover $\mathcal{M}_g$ by a finite collection of contractible sets using Fenchel--Nielsen coordinates.
\medskip 

\begin{figure}[H]
\includegraphics[width=4in]{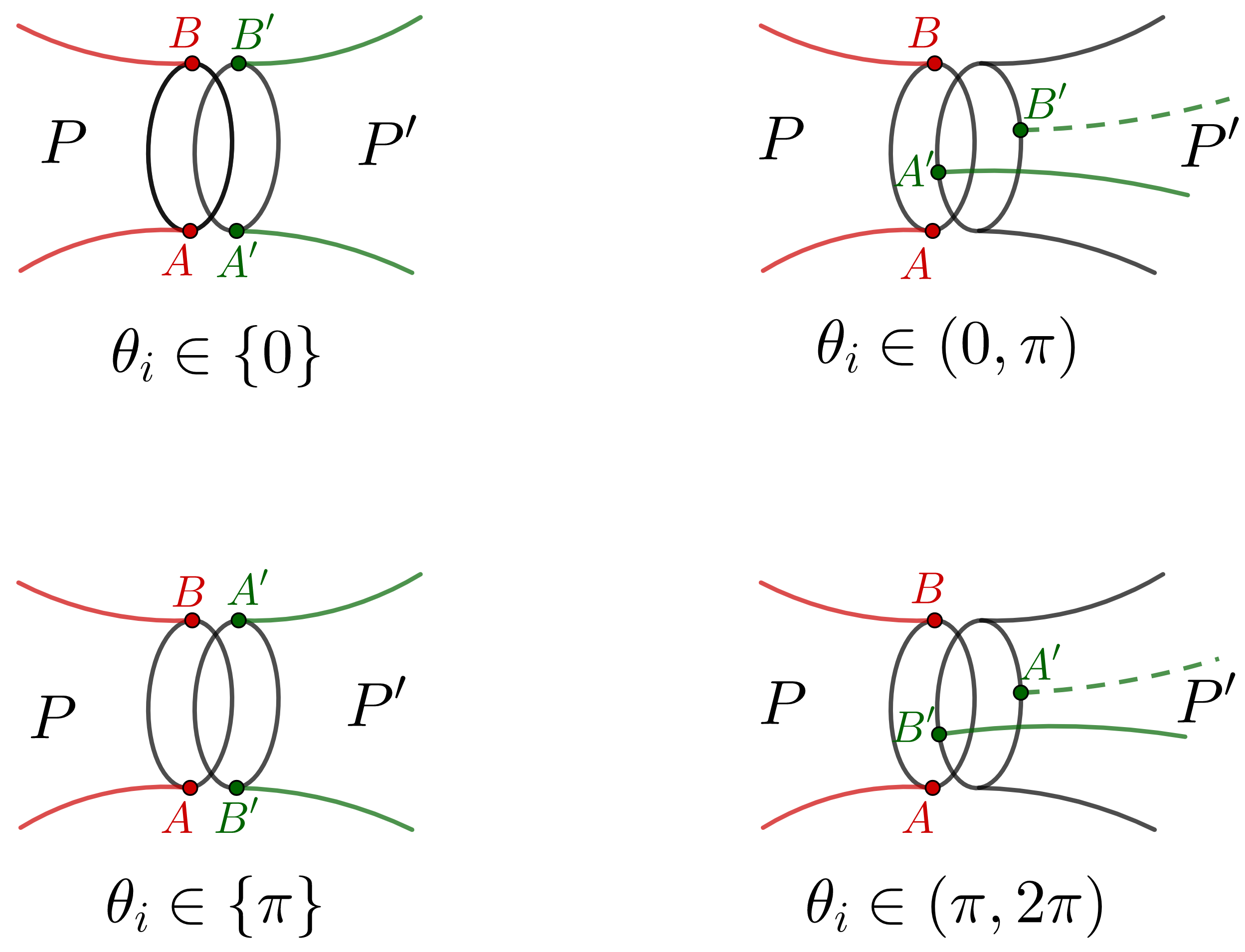}
\caption{The four possible configurations of the distinguished points $A,B$ and $A',B'$ which result from gluing two pairs of pants $P$ and $P'$ along a cuff.}
\label{glue}
\end{figure}

Let $R$ be a hyperbolic pair of pants decomposition of a Riemann surface of genus $g$, together with
an ordering of the $3g-3$ simple closed curves $(\gamma_1,\dots,\gamma_{3g-3})$ forming
the cuffs. The {\it Fenchel-Nielsen chart} on $\mathcal{M}_g$ associated to $R$ is the map $\R_+^{3g-3}\times (S^1)^{3g-3}\to \mathcal{M}_g$ sending $((\ell_1,\dots,\ell_{3g-3}),(\theta_1,\dots,\theta_{3g-3}))$ to the Riemann surface built from pairs of pants with cuff lengths given by the $\ell_i$ and glued together using the twist parameter $\theta_i$ along the $i$th cuff, see \cite[Section 10.6]{farb-margalit}.

Ranging over all possible topological types $\{R_k\}$ of pair of pants decompositions, Theorem \ref{pair-of-pants} implies that we can
cover $\mathcal{M}_g$ by a finite number of contractible sets of the form $$W_i:=(0,B_g]^{3g-3}\times (U_1\times \cdots \times U_{3g-3})$$ where each $U_j\subset S^1$ is a subset of one of the following four forms: $\{0\}$, $(0,\pi)$, $\{\pi\}$, $(\pi,2\pi)$. These four forms correspond to the four gluing configurations of the distinguished points on a cuff, see Figure \ref{glue}.

For each chart $W_i$, choose for some $C\in W_i$ an admissible identification $\pi_1(C,*)\simeq \Pi_g$ where $*$ is one of the distinguished points on a fixed cuff.
This specifies a ``reference'' set of Procesi generators $(\delta_j)_{1\leq j\leq m}$ 
over each $W_i$. Each $\delta_j\in \pi_1(C,*)$ is homotopic to a composition of paths of the following two forms, see \Cref{fig2}:

\begin{enumerate}
    \item paths $\mu$ contained in a pair of pants, which terminate at distinguished points on the cuffs, and 
    \item paths $\nu$ circling around the cuff which connect two distinguished points coming from adjacent pairs of pants.
\end{enumerate}

\begin{figure}[H]
\includegraphics[width=5in]{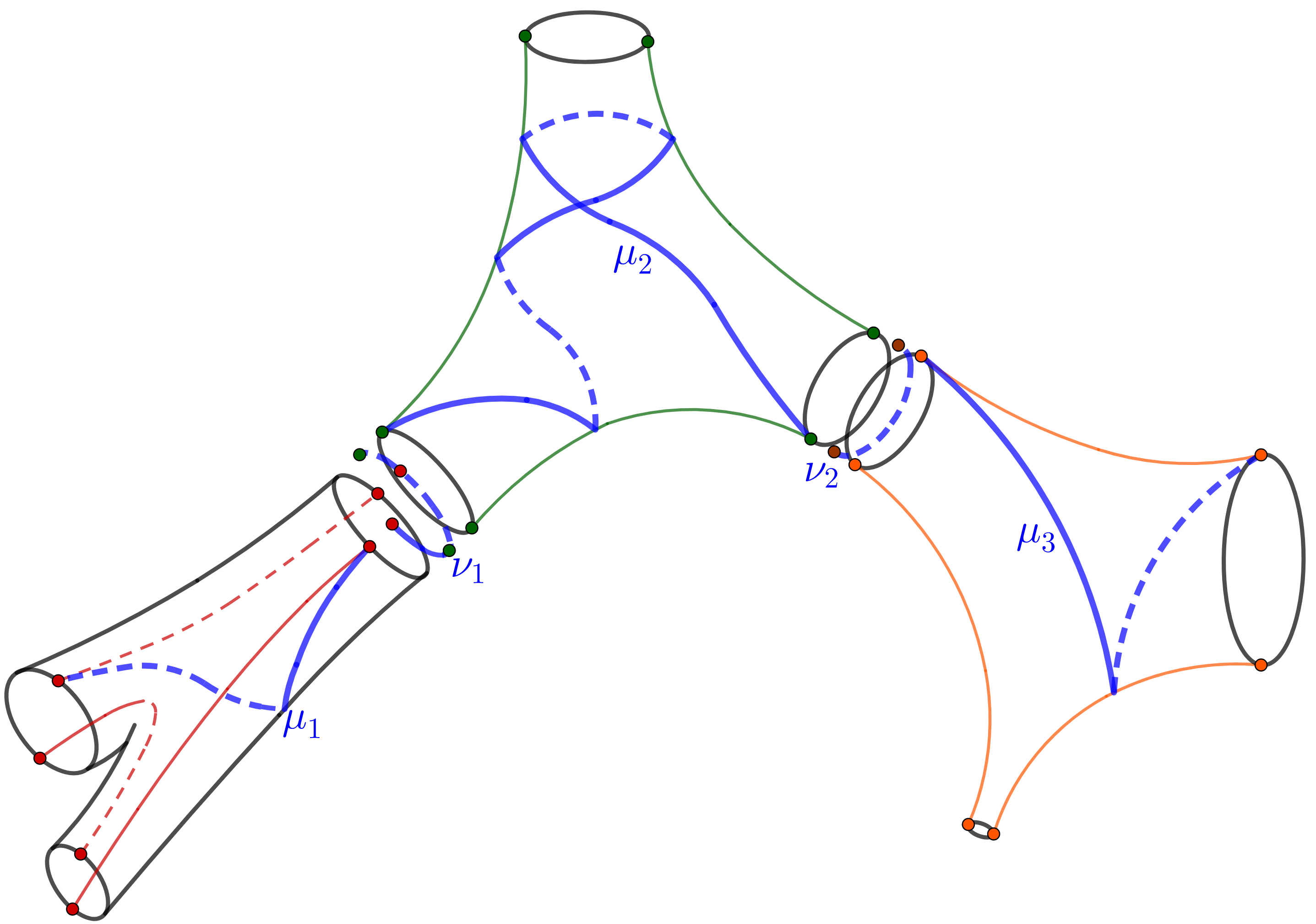}
\caption{Seam geodesics on three pairs of pants in red, green, orange, with distinguished points on the cuff in the same color. The decomposition of a loop $\delta$ into paths $\mu$ in pants and $\nu$ in cuffs, depicted in blue.}
\label{fig2}
\end{figure}

Furthermore, the relative homotopy classes of the $\mu$ and $\nu$ paths can be identified over all of $W_i$. With this geometric
set-up, we proceed to our main
theorem:

\begin{theorem} Up to admissible identification and
conjugation, there are only finitely many $\Q$-anisotropic
representations $\rho\colon \Pi_g\to \GL_n(\Z)$
which underlie a $\Z$-PVHS on some curve in $\cM_g$.
\end{theorem}

\begin{proof}
Let $\Phi\colon C\to \Gamma\backslash \bD$ be the period map
of a $\Z$-PVHS of rank $n$ on some curve $C\in \cM_g^{<\epsilon}$. Take a Bers pair-of-pants 
decomposition of $C$ as in Theorem
\ref{pair-of-pants}, realizing $C\in W_i$ as
an element of one of the above open sets $W_i$ covering
$\mathcal{M}_g$.

On all $C\in W_i$, we have a collection
of representatives of Procesi generators $\delta_j$
which decompose into of paths as in Figure \ref{fig2}.
Applying Proposition \ref{def-perim}, we may decompose
each generator $\delta_j$, up to homotopy, into three types of
paths, see Figure \ref{homotopy}:
\begin{enumerate}
    \item paths in a fixed homotopy
    class $\mu$ relative to two distinguished points
    on a truncated pair of pants $P^o(\ell_1,\ell_2,\ell_3)$ with $\ell_i\leq B_g$, 
    \item transverse geodesics on a half-collar of perimeter $C_g$, and
    \item paths winding around a cuff,
    in a fixed homotopy class $\nu$ relative to two distinguished points coming from opposite pairs of pants.
\end{enumerate}

\begin{figure}
\includegraphics[width=4in]{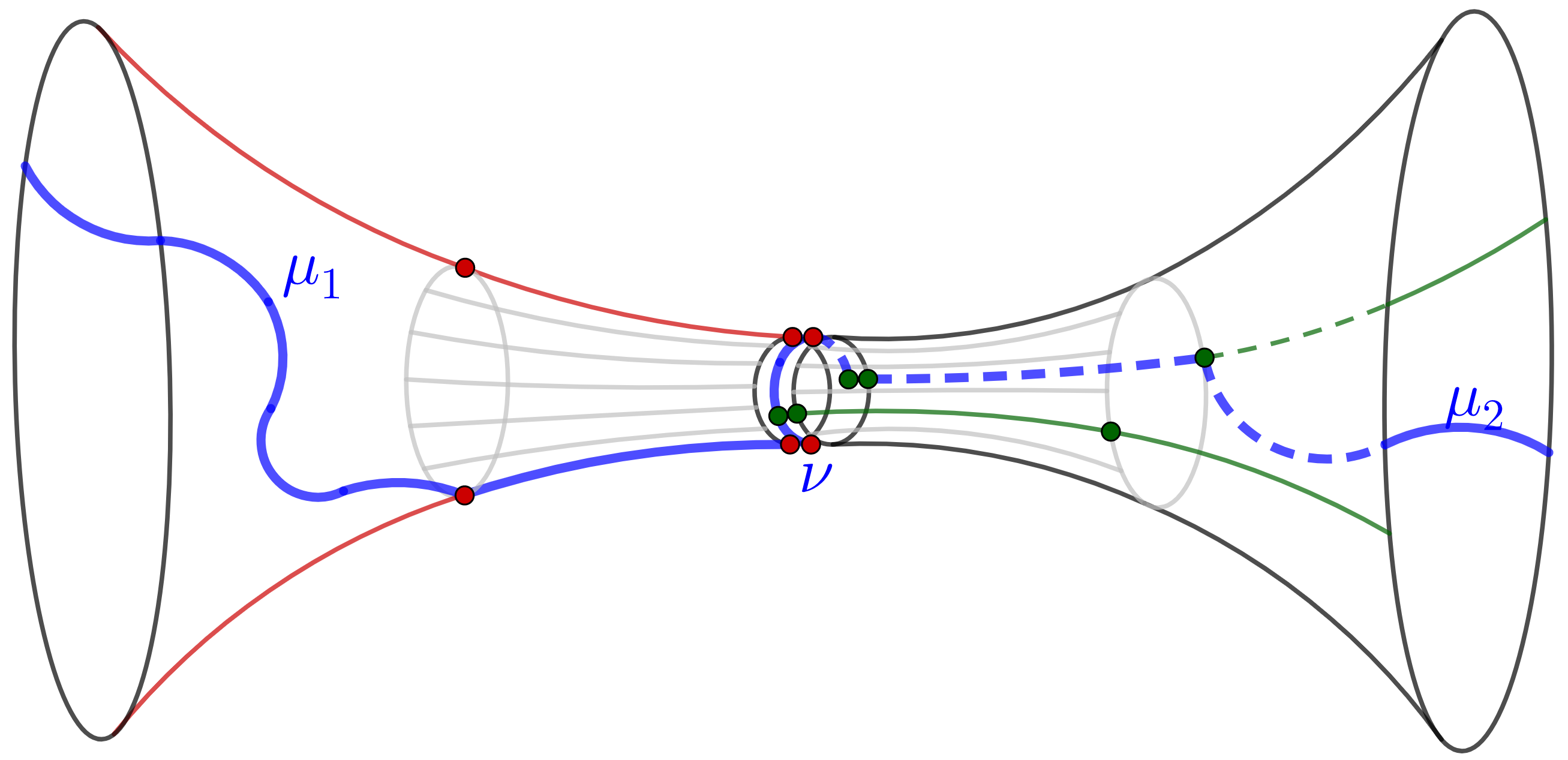}
\caption{Homotopy of the representatives of $\gamma_j$.}
\label{homotopy}
\end{figure}

Let $\widetilde{\Phi}\colon \widetilde{C}\to \bD$ be the
lift of the period map to the universal cover of $C$
and let $[0,1]$ be the lift of the loop $\delta_j$
to a path in $\widetilde{C}$. Then
$$d_\bD(\widetilde{\Phi}(0),\widetilde{\Phi}(1))\leq \!\!\!\!\! \sum_{\substack{ \textrm{paths in} \\ \textrm{truncated pants}}} \!\!\!\!\! D_\mu +\sum_{\substack{ \textrm{paths in} \\ \textrm{cuffs}}} L_\nu + 2e\max\{B,B'\}\textrm{ where}$$
\begin{enumerate}
    \item $D_\mu$ bounds the length of a representative of a relative homotopy class $\mu$ in the truncated pairs of pants (Prop.~\ref{bounded-truncated}),
    \item $L_\nu=B_g\cdot \textrm{winding}(\nu)$ bounds the length of the geodesic representing $\nu$ purely in terms of the relative homotopy class,
    \item $B$ bounds the radius of a ball covering the image of a collar
(Prop.~\ref{collar-bound}) whose core curve has length less than $\epsilon$,
\item $B'$ bounds the length of a transverse geodesic on a half-collar with core curve of length at least $\epsilon$ and perimeter $C_g$, and 
\item $e$ is the total number of 
collars crossed.
\end{enumerate}

Thus 
$d_{\bD}(\widetilde{\Phi}(0),\widetilde{\Phi}(1))$
is bounded.
We conclude by Lemma \ref{translation-bound} that in turn,
the trace ${\rm tr}(\rho(\delta_j))$ is bounded.
Then, the theorem follows as in
Proposition \ref{deligne-thick}.
\end{proof}

\begin{corollary}\label{monodromy-finite}
Let $\cS$ be a smooth connected quasi-projective complex algebraic variety and let $\pi:\cY\to \cS$ be a smooth projective morphism. There are only finitely representations $\pi_1(Y_0,*)\rightarrow 
\GL_n(\Z)$, up to conjugacy, which underlie a $\Z$-PVHS with $\Q$-anisotropic monodromy on some fiber $Y_s$ of $\pi:\cY\rightarrow\cS$, after an identification $\pi_1(Y_0,*)\simeq\pi_1(Y_s,*)$ induced by
moving $*$ in $\cY$.
\end{corollary}

\begin{proof} This follows from the discussion at the beginning of the section, using the Lefschetz hyperplane theorem. 
\end{proof}

\section{Douady spaces of polarized distribution manifolds}

In this section we abstract some key elements of Hodge
manifolds, especially in the case where $\Gamma$ is cocompact.

\begin{definition} A {\it distribution manifold} $(X,\Xi)$ is a compact,
complex manifold $X$, together with a holomorphic subbundle
$\Xi\subset TX$ of its tangent bundle (i.e. a holomorphic 
distribution).\footnote{We do not require the distribution to be integrable.}

Let $L\to X$ be a holomorphic line bundle and let $h$ be a Hermitian 
metric on $L$. We say that $(L,h)$ is {\it positive} on $(X,\Xi)$ if 
the $(1,1)$-form 
$\omega_L:=\frac{i}{2\pi}\partial\overline{\partial}\log h$ 
satisfies $\omega_L\big{|}_{\Xi}>0$. We call $(L,h)$ a {\it polarization}
of the distribution manifold $(X,\Xi)$. \end{definition}

We now recall fundamental results on the analogues of
the Hilbert and Chow varieties for complex manifolds and 
analytic spaces.

\begin{definition} An {\it analytic cycle} on $X$ is a finite
formal $\Z$-linear combination $\sum_i n_i[Z_i]$ of 
irreducible, closed, reduced analytic subspaces
$Z_i\subset X$ of a fixed dimension. 
An analytic cycle is {\it effective} if $n_i\geq 0$. \end{definition}

We have then the following fundamental result of Barlet, 
see \cite{barlet}.

\begin{theorem}\label{barlet} The effective analytic 
cycles on $X$ are parameterized by a complex
analytic space.
\end{theorem}

We call a connected component $\mathfrak{B}$ 
of this analytic space a {\it Barlet space}.
Unlike the Hilbert
scheme, the connected components
may have infinitely many irreducible
components, see Remark
\ref{hironaka}.

\begin{definition} Let $(X,\Xi)$ be a distribution manifold. A {\it horizontal Barlet space} $\mathfrak{B}^{\Xi}$ of $(X,\Xi)$ is a connected component
of the sublocus of $\mathfrak{B}$ defined by the following property: \begin{align*} \textstyle \sum_i n_i&[Z_i]\in
\mathfrak{B}^{\Xi}\textrm{ iff there is a dense open set }  \\ & Z^o\subset \cup_iZ_i\textrm{ for which }TZ^o\subset \Xi. \end{align*} \end{definition}

This is visibly a locally closed analytic
condition on the Barlet space. In fact, much more is true:

\begin{theorem}\label{properness} Let $(X,\Xi,L,h)$ be a polarized distribution
manifold. Any horizontal Barlet space $\mathfrak{B}^{\Xi}$ is
a proper analytic space.

Furthermore, there are only finitely many Barlet spaces parameterizing cycles of pure codimension $d$ on which $c_1(L)^{n-d}$ is bounded.\end{theorem}

\begin{proof} Let $g$ be an arbitrary hermitian metric on $X$,
for instance, we can construct $g$ via a partition of unity.
Define a smooth distribution $\Xi^{\perp}\subset TX$ by
$\Xi^{\perp}_x := (\Xi_x)^{\perp g}$. Then, we have a $g$-orthogonal splitting 
$TX = \Xi\oplus \Xi^{\perp}$ as smooth $\C$-vector bundles.
Let $g^\perp$ denote the
degenerate, semi-positive hermitian form on $TX$ which is
defined by $(0,g\big{|}_{\Xi^{\perp}})$ with respect the
decomposition $TX =\Xi\oplus \Xi^{\perp}$.

Let $N>0$ and define a symmetric tensor by $$\widetilde{g}(v,w):=
\omega_L (v,Jw)+ Ng^\perp(v,w)\in S^2T^*X.$$
We claim that $\widetilde{g}$
is a Hermitian metric on $X$ for sufficiently large $N$. This 
follows from $\omega_L(v ,Jw)$
being positive-definite on $\Xi$, $g^\perp$ vanishing on $\Xi$ and
being positive definite on $\Xi^{\perp}$, and compactness of $X$.

For any codimension $d$ analytic cycle 
$Z:=\sum_i n_i[Z_i]\in \mathfrak{B}^{\Xi}$, define $$\vol_L(Z) 
= \sum_i n_i\int_{Z_i} c_1(L)^{n-d} = [Z]\cdot c_1(L)^{n-d}.$$
Observe that $c_1(L)^{n-d}$ is pointwise positive on $Z_i^o\subset Z_i$. 
Furthermore $\vol_L(Z)$ is constant
on a connected component of $\mathfrak{B}^{\Xi}$ because it is given
as the intersection number on the right. Next,
we define $$\vol_{\widetilde{g}}(Z):=\sum_i n_i \int_{Z_i} 
\vol_{\widetilde{g}|_{Z_i}}$$ and observe 
$\vol_{\widetilde{g}}(Z)=\vol_L(Z)$ because 
$\widetilde{g}(\cdot,\cdot)\big{|}_{\Xi} = \omega_L(\cdot,J\cdot)\big{|}_{\Xi}$ and $TZ_i^o\subset \Xi$.
Thus, $X$ admits a hermitian metric $\widetilde{g}$ in which
$\vol_{\widetilde{g}}(Z)$ is constant on a connected component of
$\mathfrak{B}^{\Xi}$, equal to $[Z]\cdot c_1(L)^{n-d}$.

Let $Z^{(1)},Z^{(2)},\cdots$ be a countable sequence of effective
analytic cycles in (possibly different) connected components 
$\mathfrak{B}^{\Xi,(i)}$, for which $\vol_L=\vol_{\widetilde{g}}$
remains bounded. By a theorem of Harvey-Schiffman 
\cite[Thm.~3.9]{harvey-shiffman},
we can extract a convergent subsequence
that converges to an effective analytic cycle $Z^{(\infty)}$
for which $\vol_{\widetilde{g}}(Z^{(i)})$ converges to 
$\vol_{\widetilde{g}}(Z^{(\infty)})$. Such convergence
defines the topology on $\fB$.

By \cite[Prop.~2.3]{fujiki}, the $Z^{(i)}$ converge in the
sense of currents of integration to 
$Z^{(\infty)}$, and in particular, the integrals
$\int_{Z^{(i)}}\omega_L^{n-d}$ must converge
to $\int_{Z^{(\infty)}} \omega_L^{n-d}$ and so remain bounded.
Additionally, we have $\vol_{\widetilde{g}}(Z^{(\infty)})=
\vol_L(Z^{(\infty)})$ and this
equality holds for any choice $N$ in the definition
$\widetilde{g} = \omega_L+Ng^\perp$. We conclude that
there is a Zariski-dense open subset $Z^o\subset Z^{\infty}$
for which $TZ^o\subset (\Xi^{\perp})^{\perp}=\Xi$,
as otherwise $\vol_{\widetilde{g}}(Z^{(\infty)})$ would
increase as $N$ increases.

Thus, the union of all components 
$\mathfrak{B}^{\Xi}$ for which $c_1(L)^{n-d}$
is bounded is sequentially compact. Hence each component of $\mathfrak{B}^{\Xi}$  is a compact analytic space,
and there are only finitely many
components with bounded $\vol_L$. The theorem follows.
\end{proof}

\begin{remark}\label{hironaka}
In general, a Barlet space 
of a compact analytic 
space $X$ need not have finitely
many irreducible components, even if
$X$ is a smooth, proper $\C$-variety.
A famous counterexample is due to Hironaka: let
$C,D\subset M$ be two smooth curves in a smooth projective $3$-fold
$M$, with $C\cap D = \{p,q\}$. We can consider the variety
$$\hM:=Bl_{\hC} Bl_D (M\setminus q)\cup Bl_{\hD} Bl_C(M\setminus p),$$
that is, we blow up $M$ along $C$ and $D$, but in opposite orders
at $p$ and $q$. If $F$ is a fiber of one of the exceptional
divisors, the Barlet space containing $F$ is an infinite chain of curves:
$F$ admits a deformation to a cycle of the form $F+(Z_1+Z_2)$
where $Z_1$ and $Z_2$ are the strict transforms of the fibers
at $p$ and $q$ of the first blow-up in the second blow-up.
We may further deform to $F+n(Z_1+Z_2)$ 
for any $n\in \N$.

As a compact
complex manifold, $\hM$ admits
a hermitian metric $h$. Following Theorem
\ref{properness}, one many consider
the space of  
$d$-cycles $Z$ of bounded volume 
${\rm vol}_h(Z)\leq C$,
and indeed this is compact.
But it is only semi-analytic---for example
as $F$ deforms in its connected component of $\fB$, the volume
will increase until one hits the
``cut-off'' $C$. So
the compact ${\rm vol}_h(Z)\leq C$ Barlet space 
is only semianalytic.
This does not present an issue when 
$h$ is associated to a closed $2$-form,
i.e. defines a K\"ahler
metric, because the volume is then
locally constant on $\fB$. Indeed, this
is the key point in Fujiki's work
\cite[Proof of Prop.~4.1]{fujiki}.
Theorem \ref{properness} is a generalization
of the same principle to distribution manifolds.
\end{remark}

We now consider the analogue of Hilbert spaces. A {\it Douady space} of $X$ is an
analytic space $\fD$ parametrizing flat 
families of closed analytic subspaces of $X$, see \cite[\S 9.1]{douady} for a precise definition. By the main theorem of Douady \cite[pp.~83-84]{douady}, there is a universal analytic subspace
$\mathcal{Z}\subset \fD\times X$ which is flat over $X$,
and any flat family parameterized by a base $M$
is the pullback along an analytic classifying morphism $M\to \fD$.

Given a sub-analytic space $Z\subset X$,
we can define an effective analytic cycle $[Z]\in \fB$
called the {\it support}.
It is the positive linear combination $\sum_{i} n_i[Z_i]$ where $Z_i$
are the irreducible components of the reduction of $Z$ that have
top-dimensional set-theoretic support,
and $n_i$ is the generic order of
non-reducedness of $Z$ along $Z_i$,
see \cite[Sec.~3.1]{fujiki}. There is an analogue, the 
{\it Douady-Barlet morphism} $[\cdot ]\colon \fD\to \fB$, of the Hilbert-Chow morphism, sending an analytic space to its support.

\begin{theorem}[{\cite[Prop.~3.4]{fujiki}}]\label{fujiki}
Suppose that $\fB$ is a compact 
analytic subspace of the Barlet space.
Then, the Douady-Barlet morphism is proper, 
on each component $\fD$ of the Douady space, 
of analytic spaces whose support
$[\cdot]$ lies in $\fB$. 
\end{theorem}

\begin{proof}
As stated, \cite[Prop.~3.4]{fujiki} only applies 
when $\fB$ is a compact irreducible component
of the Barlet space, but the
exact same proof applies to any
compact analytic subspace of the Barlet space.
\end{proof}

\begin{definition}A {\it horizontal Douady space} $\fD^{\Xi}$ is a connected component of
the sublocus of $Z\in \fD$ for
which $[Z]\in \fB^{\Xi}.$ \end{definition}

\begin{remark} It is important to note that the Zariski tangent
space of $Z\in \fD^{\Xi}$ is not required to lie in $\Xi$.
For instance, consider a flat family $\cZ^* \to C^*=C\setminus 0$
of complex submanifolds of 
$X$, with the tangent bundle $T\cZ_t$ lying in $\Xi$
for all $t\in C^*$. The flat limit $Z_0$
over the puncture might be nilpotently thickened in directions outside 
of $\Xi$, if the total space of the family itself does not have a tangent 
bundle $T\cZ^*$ lying in $\Xi$, and this could even occur generically along $Z_0$.
\end{remark}

\begin{corollary} Let $(X,\Xi,L,h)$ be a polarized distribution
manifold. Then, each connected component of $\fD^{\Xi}$ is a proper
analytic space. \end{corollary}

\begin{proof} This follows directly from Theorem \ref{fujiki}
and Theorem \ref{barlet}. \end{proof}

\begin{theorem}\label{griffiths-algebraic}
Let $Z\in \fD^{\Xi}$ lie in a horizontal Douady space.
Then $Z$ is projective, and $L\big{|}_Z$ is an ample line bundle.
\end{theorem}

\begin{proof} 
A simplification of the
proof in \cite[Thm.~1.1]{bakker-brunebarbe-tsimerman}
applies. It follows from Siu and Demailly's resolution
\cite{siu1, siu2, demailly} of the 
Grauert-Riemenschneider conjecture, applied
to a resolution of $Z$, that $Z$ is Moishezon.
Next, we have:

\begin{lemma}\label{tan-strata}
Let $S$ be an smooth, locally-closed
stratum of the Whitney stratification
of $\cup_i Z_i$. Then $TS\subset \Xi$. \end{lemma}

\begin{proof} By assumption, there is a dense open $Z^o\subset \cup_i Z_i$
for which $TZ^o\subset \Xi$. We claim that $TS\subset \overline{TZ^o}$
lies in the Zariski closure of $TZ^o$ in $TX$. Then the result will follow 
as $\Xi$ is Zariski-closed in $TX$. 

Let $Z_i$ be an irreducible component containing $S$. Consider the map
$d\pi_i\colon T\wZ_i \to TX$ from a resolution. Let $\wZ_i^o:=\pi_i^{-1}(Z_i\cap Z^o).$ As $d\pi_i$ is continuous
and $d\pi_i(T\wZ_i^o)\subset TZ^o$, we have ${\rm im}(d\pi_i)\subset \overline{TZ^o}$.
The claim follows if we can show ${\rm im}(d\pi_i)\supset TS'$, for a dense
open $S'\subset S$, i.e. can we
lift a generic tangent vector of $S'$ to $\wZ_i$.
This follows from the generic smoothness of
$\pi_i\big{|}_{\pi_i^{-1}(S)^{\rm red}}$.
\end{proof}

Lemma \ref{tan-strata} implies that we have $L^d\cdot V>0$
for any subvariety $V$ of dimension $d$, because
$TV$ is generically contained
in the tangent bundle of some singular stratum $S$
and $\frac{i}{2\pi}\partial \overline{\partial} \log(h)$
is positive definite on $\Xi$. So $Z$ satisfies
the Nakai-Moishezon criterion. Then, a theorem of
Koll\'ar \cite[Thm.~3.11]{kollar}
implies that $Z$ is projective.
\end{proof}

\begin{definition} Let $\fC\subset (\fD^{\Xi})^{\rm red}$
be an irreducible component of a horizontal Douady space.
For $Z_t\in \fC$ let $L_t:=L\big{|}_{Z_t}$.

We say that $\fC$ is
{\it locally $L$-determined} 
if there exists an analytic open
set $U\subset \fC$ for which 
$(Z_s,L_s)\not\simeq (Z_t,L_t)$
for all $s,t\in U$, $s\neq t$.
\end{definition}

\begin{theorem}\label{main1} Let $\fC$ be 
an irreducible component of the reduction
of the horizontal Douady space of $(X,\Xi,L,h)$,
which is locally $L$-determined.
Then $\fC$ is Moishezon.\end{theorem}

\begin{proof} Let $u\colon \fZ^{\Xi}\to \fC$ be the universal flat family
and let $\fL\to \fZ^{\Xi}$ be the universal polarizing line bundle.
For any fixed $n\in \N$, the locus $\fC_n\subset \fC$ of projective
(Thm.~\ref{griffiths-algebraic}) schemes $Z\in \fC$
on which $nL=n\mathfrak{L}\big{|}_Z$ is not very ample is closed.
Taking the sequence
 $$\cdots\subset \fC_{3!}\subset \fC_{2!}\subset \fC_{1!}\subset \fC$$
gives a nested sequence of closed analytic subspaces.
The intersection is empty since for all $Z\in \fC$, there is
some $n_Z\in \N$ for which $n_ZL$ is very ample, and $n_Z\mid i!$
for all $i\geq n_Z$. We conclude some $\fC_n$ is empty
for large enough $n$, so $|nL|$ is a projective embedding
for all $Z\in \fC$.

Furthermore, the locus on which $H^i(Z,nL)$ jumps in dimension
is also closed, and so by the same argument, we may assume
$h^i(Z,nL)=0$ for all $i>0$ and all $Z\in \fC$. Then $u_*(n\fL)$
is a vector bundle of rank $$N+1:=\chi(Z,nL)= h^0(Z,nL).$$ It is a vector
bundle because $\chi$ is constant in (analytic) flat families.

Let $\bP\to \fC$
be the projective frame bundle of $u_*(n\fL)$, a principal
holomorphic $J=\PGL(N+1)$-bundle. Points of $\bP$ correspond
to some $Z\subset X$, and a basis of sections of
$H^0(Z,nL)$, modulo scaling. We
have an analytic map $$\phi\colon \bP\to \cH$$ where
$\cH\subset {\rm Hilb}(\bP^N)$ is the component of the Hilbert
scheme with Hilbert polynomial $\chi$, sending
$(Z,[s_0:\cdots:s_N])\in \bP$
to the closed subscheme of $\bP^N$ with the given embedding.
Note $\cH$ is projective and $\phi$ is equivariant 
with respect to
the natural $J$-action on both sides. 

We have assumed that
$\fC$ is locally $L$-determined: 
There exists
some analytic open $U\subset \fC$ for which 
$(Z_s,L_s)\not\simeq (Z_t,L_t)$
for all $s,t\in U$, $s\neq t$. This implies that $(Z_t,nL_t)\not\simeq (Z_s,nL_s)$
for all $s\neq t$ in a possibly smaller neighborhood.
Thus, the $J$-orbits in $\cH$ 
corresponding to $(Z_t,nL_t)$ are all distinct
in an analytic open set. 
We now apply Lemma \ref{orbit-mero} below to conclude that $\fC$ is Moishezon.
\end{proof}

\begin{lemma}\label{orbit-mero}
Let $\fC$ be a compact, complex manifold
and let $p \colon \bP\to \fC$ be a principal $J$-bundle, 
for $J$ a complex algebraic group. Let $\cH$ be a projective
variety with an algebraic $J$-action and suppose that
$\phi\colon \bP\to \cH$ is a $J$-equivariant holomorphic
map, such that there exists an analytic open
set $U\subset \fC$ for which the map
\begin{align*} 
\fC &\to \cH/J \\
u &\mapsto \phi(p^{-1}(u)) \end{align*}
is injective on $U$. Then $\fC$ is Moishezon.
\end{lemma}

We view $\cH/J$ as a set of $J$-orbits in the above lemma,
as it may not have the structure of an algebraic variety.

\begin{proof}

Consider the locus of orbit closures
$\mathrm{O}:=\{\overline{J\cdot x}\,\big{|}\,x\in \cH\}
\subset {\rm Chow}(\cH)$, viewed as pure-dimensional
cycles on $\cH$. Note that the $J$-orbit closures
will in general have different dimensions and may lie
in different components of the Chow variety of $\cH$.
A point $
\overline{J\cdot x}\in \mathrm O$ uniquely 
determines a $J$-orbit, 
since a $J$-orbit is uniquely recovered from the 
closure of the orbit of a general point 
$x'\in \overline{J\cdot x}$.

Since the action of $J$ is algebraic on $\cH$, the space 
$\mathrm O$ is stratified by algebraic varieties
$$\mathrm O = \mathrm O_1\sqcup \cdots \sqcup \mathrm O_m$$ with each $\mathrm O_j$ an irreducible,
locally closed set of some component ${\rm Chow}_j(\cH)$ of the Chow variety.
Let $\cH_j\subset \cH$ be the locally closed set of points $x\in \cH$ for which
$\overline{J\cdot x}\in \mathrm O_j$.

Observe that $\cH = \cH_1\sqcup \cdots \sqcup \cH_m$ 
is a Zariski locally closed, $J$-invariant 
stratification of $\cH$. Pulling back this stratification
along $\phi$ gives a $J$-invariant stratification
of $\bP$, which in turn descends along $p$ to
an analytic Zariski locally closed
stratification of $\fC$. Thus, there is
an analytic Zariski closed set
$\fC'\subset \fC$ such that
$\phi(p^{-1}(\fC\setminus \fC'))\subset \cH_j$ 
for some stratum $\cH_j$.

Since $\bP$ is irreducible, we have $\phi(\bP)\subset \overline{\cH}_j$.
Observe that there is a rational map (a morphism on $\cH_j$)
\begin{align*} \psi\colon \overline{\cH}_j&
\dashrightarrow \overline{\mathrm O}_j \\
x&\mapsto \overline{J\cdot x}\end{align*}
with the closure of the latter taken in ${\rm Chow}_j(\cH)$,
which is projective.

Let $V\subset \fC$ be a small, analytic open 
chart around any point in $\fC$.
There is a local analytic section of $\bP\big{|}_V\to V$,
call it $s_V$.
Then, $\phi\circ s_V\colon V\to \overline{\cH}_j$
is analytic and $\psi$ is rational, so the composition
$$\psi \circ \phi\circ s_V\colon 
V\dashrightarrow {\mathrm \oO}_j$$ is a meromorphic map.
Furthermore, since $\psi$ collapses $J$-orbits, 
and $\phi$ is $J$-equivariant, we conclude that this local 
meromorphic map is independent of choice of local
section $s_V$.
So these maps patch together to give a global
meromorphic map
$\alpha\colon \fC \dashrightarrow {\mathrm \oO}_j.$

Since $\alpha$ is meromorphic, by Hironaka,
there is a resolution of indeterminacy
$$\fC\xleftarrow{\beta} \widetilde{\fC}\xrightarrow{\gamma} 
{\mathrm \oO}_j$$ 
of $\alpha=\gamma\circ \beta^{-1}$ with $\beta$
bimeromorphic. Finally, for the analytic
open set $U$ in the statement of the lemma, we
have that the holomorphic map
$$\alpha\vert_{U\setminus \fC'}\colon 
U\setminus \fC'\to {\rm O}_j$$ is injective,
because $\cH_j/J={\rm O}_j$ (e.g.~as sets). We deduce
that the morphism $\gamma$ is generically finite
onto its image, which being closed in the projective
variety ${\rm \oO}_j$ is projective.
As the Stein factorization of $\gamma$ 
is finite over the image of $\gamma$, 
it is projective.
So $\fC$ is bimeromorphic to
a projective variety.
\end{proof}

\begin{remark} The assumption that $\fC$ is locally $L$-determined is necessary.
For instance, let $X$ be an arbitrary compact, complex manifold,
and consider the distribution manifold for which $\Xi=0$. It admits
a polarization by setting $L=\cO_X$ with $h$ the trivial metric.
Then, the Douady space of points in $X$ is a horizontal Douady space,
isomorphic to $X$ itself. But of course, $X$ need not be Moishezon,
so not all horizontal Douady spaces are Moishezon in this generality.
\end{remark}

\begin{metadefinition}\label{data of GAGA type} We define {\it data of GAGA type} 
on $X$ to be a collection of holomorphic data ${\rm Data}_X$
to which the GAGA theorem applies, upon restriction
to a projective scheme $Z\in \fD^{\Xi}$.
\end{metadefinition}

\begin{example} An example of data of GAGA type would be
${\rm Data}_X = (F^\bullet, \nabla)$ where $F^\bullet$
is a descending filtration of holomorphic vector bundles on
$X$ and $\nabla$ is a holomorphic connection on $F^0$.
For any horizontal analytic space $Z\in \fD^{\Xi}$,
the restriction of $F^\bullet$ to $Z$ is a filtration
$F^\bullet_Z$ of algebraic vector bundles, by Serre's GAGA theorem \cite{gaga}.

Similarly, GAGA
holds for vector bundles with flat
connection, by interpreting
flat connections as splittings
of the Atiyah sequence
\cite[I.2.3]{deligneed}. In
particular, 
the restriction of $\nabla$ to a
connection $\nabla_Z$ on $F^0_Z$
is an algebraic connection.
\end{example}

\begin{metatheorem}\label{meta} Let ${\rm Data}_X$ be data of GAGA type on $X$.
We say that an irreducible, reduced, closed
analytic subspace $\fD_0\subset \fD^{\Xi}$
is {\rm locally ${\rm Data}_X$-determined} 
if the isomorphism type of the restriction 
of this data to $Z\in \fD_0$ is determinative
in some analytic open set $U\subset \fD_0$:
$(Z_s,{\rm Data}_s)\not\simeq (Z_t,{\rm Data}_t)$ for all $s\neq t\in U$.

Then Theorem \ref{main1} still holds: $\fD_0$ is Moishezon.
\end{metatheorem}

\begin{proof}[Sketch.] By GAGA, the restriction of 
${\rm Data}_X$ to any
$Z\in \fD_0$ is algebraic data, denoted ${\rm Data}_Z$.
The general form of such algebraic data, together with $Z$,
is parameterized by an algebraic variety (adding rigidifying
data corresponding to an algebraic group action as necessary),
admitting an algebraic compactification
$\cH_{\rm Data}$. Then, we apply the same argument as
in Theorem \ref{main1} to the classifying map
\begin{align*} \fD_0&\dashrightarrow \cH_{\rm Data} \\
Z&\mapsto (Z,{\rm  Data}_Z)\end{align*}
to conclude that $\fD_0$ is Moishezon.
\end{proof}

\begin{corollary} \label{cvhs-embed}
Let $(X, \Xi, L, h)$ be a polarized
distribution manifold, endowed with
a tuple
$(F_i^\bullet, \nabla_i)_{1\leq i\leq k}$
of filtered flat vector bundles
on $X$.
Suppose that $\fD_0\subset \fD^\Xi$
is an irreducible, reduced,
closed analytic subspace of a horizontal
Douady space of $X$, which is locally
$(F_i^\bullet, \nabla_i)_{1\leq i\leq k}$-determined.
Then, $\fD_0$ is Moishezon.
\end{corollary}

\begin{proof}
The corollary is an instance
of Meta-Theorem \ref{meta}.
For the sake of explicitness, we will
concretely construct a compactification
$\cH_{\rm Data}$ of the parameter space
of relevant data of GAGA type.

Denote by $\pi:\fZ\rightarrow \fD_0$ 
the pullback of the universal flat 
family over $\fD^\Xi$ and 
$f:\fL\rightarrow \fZ$ the universal 
polarizing line bundle.  

For convenience of exposition,
we begin with just one filtered flat vector bundle
$(F^\bullet, \nabla)$ on $X$.
Let $\cH$ be the component of the Hilbert scheme that
$|nL|$ maps $Z$ into. Let 
$\pi_\cH\colon \fZ_\cH\to \cH$ be the universal
flat family over $\cH$.

The Hilbert polynomials
$P^\bullet$ of the vector bundles $F^\bullet_Z$
which arise from restricting $F^\bullet$ 
are constant along $Z\in \fD_0$ by flatness. 
We may choose
integers $m_p, n_p \gg 0$ for which any vector
bundle (even coherent sheaf)
with Hilbert polynomial $P^p$ over 
any $Z\in \cH$
is a quotient of the form $$(-m_pL)^{\oplus n_p}\onto F^p_Z.$$
For instance, choose $m_p$
uniformly over all of $\cH$ so that $F_Z^p(m_pL)$
is globally generated with vanishing higher cohomology.
Then for a fixed $n_p$, there is a surjection
$\cO_Z^{\oplus n_p}\onto F^p_Z(m_pL)$ 
corresponding to a basis of global sections.
Furthermore, this quotient is uniquely determined by the
induced surjection  $$H^0(Z, (k_pL)^{\oplus n_p})\onto 
H^0(Z, F_Z^p((m_p+k_p)L))$$ for all $k_p$ large enough.
We can ensure that $h^0(Z,k_pL)$ is constant over all of $\cH$.
So this defines an embedding of the relative moduli
space of coherent sheaves with Hilbert polynomial $P^p$ over $\cH$
into a Grassmannian bundle ${\rm Gr}(V_p)$ of the vector bundle
$V^p:=(\pi_\cH)_*(k_pL)^{\oplus n_p}$. This is the standard 
construction, due to Grothendieck \cite{grothendieck-fga-iv}, 
of an embedding of the quot-scheme into a Grassmannian,
performed relatively over $\cH$.

The inclusion $F^p_Z\hookrightarrow F^{p-1}_Z$ is an element of
$H^0(Z, (F_Z^p)^*\otimes F_Z^{p-1})$. This vector space
includes into $H^0(Z, (m_pL)^{\oplus n_p}\otimes F_Z^{p-1})$
and by choosing $m_p\gg m_{p-1}$, we can ensure that the latter
receives a surjection from $H^0(Z,(m_pL)^{\oplus n_p}\otimes (-m_{p-1}L)^{\oplus n_{p-1}}).$ Thus, the inclusion $F_Z^p\hookrightarrow F_Z^{p-1}$
is determined by an $n_p\times n_{p-1}$-matrix of
global sections of $(m_p-m_{p-1})L$, uniquely up to
a subspace of this vector space of matrices.
Choosing $k_p$ so that $k_{p-1}+m_{p-1} = k_p + m_p$ we can insure that
$F_Z^p\into F_Z^{p-1}$ is induced by an inclusion 
$V_{Z}^p\into V_{Z}^{p-1}$ of the fibers over $Z\in \cH$.

Thus, the isomorphism type of $F^\bullet_Z$
as a filtered vector bundle is determined
by
\begin{enumerate}
    \item an element in the Grassmannian
    ${\rm Gr}(V^p)$ for each $V^p$ 
    (determining $F^p$) and
    \item a collection of inclusions 
    $i_p\colon V^p\to V^{p-1}$ (determining
    $F^p\to F^{p-1})$.
\end{enumerate}
Denote by ${\rm Fl}(F^\bullet_Z)$ 
the space of all such collections. The isomorphism type of $F^\bullet_Z$ 
is uniquely determined by a
$J'$-orbit on ${\rm Fl}(F^\bullet_Z)$,
for $J'$ an algebraic group.
Concretely, $J'$ is the 
group parameterizing changes-of-basis
of $H^0(Z,F^p_Z(m_pL))$ and
changes-of-lift of the 
inclusions $F_Z^p\into F_Z^{p-1}$.

Let $\cH_{\rm filt}$ be the principal
$J'$-bundle consisting of a filtered
vector bundle $F^\bullet_Z$ on some $Z\in \cH$ with
Hilbert polynomial $P^\bullet$,
together with its rigidifying data in
${\rm Fl}(F^\bullet_Z)$. We have a forgetful
map $\cH_{\rm filt}\to \cH$.

Over $\cH_{\rm filt}$, we construct
the relative moduli space
$\cH_{(F^\bullet, \nabla)}^o\to \cH_{\rm filt}$
of algebraic connections $\nabla$ on $F^0$. Applying the above construction
for each filtered flat vector bundle 
$(F_i^\bullet,
\nabla_i)$,
we get a parameter space 
$$\cH^o_{\rm Data} = 
\cH_{(F^\bullet_i, \nabla_i)_{1\leq i \leq k}}$$
defined as the fiber product of $\cH^o_{(F_i^\bullet, \nabla_i)}$ for $i=1,\dots,k$ over $\cH$.
Finally, take an algebraic 
compactification $\cH_{\rm Data}^o\into \cH_{\rm Data}$.

As in Theorem \ref{main1}, we have a principal $J$-bundle
$\bP\to \fD_0$ with $J=\PGL(N+1)$ corresponding to changes of
basis of $H^0(Z,nL)$. Over $\bP$, 
we have a principal
$J'$-bundle $\bP'\to \bP$ 
consisting of the space of all 
rigidifying data for the tuple
$(F^\bullet_i, \nabla_i)\vert_Z$ of filtered
flat vector bundles, as above.
We also have an algebraic connection $\nabla_i\vert_Z$
on $F^0_i\vert_Z$. 
So there is a holomorphic classifying map
$\bP'\to \cH_{\rm Data}$, which is
$J'$- and $J$-equivariant for the actions
on the source and target. We may
now apply the argument of Lemma \ref{orbit-mero}
(which easily generalizes to a sequence
of principal bundles)
to conclude that $\fD_0$ is Moishezon.
\end{proof}

\section{Algebraicity of the non-abelian
Hodge locus}

We now apply the general results of the previous
section to the polarized
distribution manifold $(X_\Gamma, \Xi, L, h)$
where $X_\Gamma = \Gamma\backslash \bD$ for $\Gamma$
cocompact torsion free, $\Xi$ is the Griffiths
distribution, $L$ is the Griffiths line bundle, and $h$ is the 
equivariant hermitian metric. Let 
$G=G_1\times \cdots \times G_k$
be the decomposition of the semisimple group
$G=\bfG^{\rm ad}(\R)^+$ into $\R$-simple factors.
These give the $\C$-simple factors of $G_\C$
by \cite[4.4.10]{simpsons-ihes-local-systems}.

We have a decomposition
$\bD = \bD_1\times\cdots \times \bD_k$
and on each factor $\bD_i$ we have a filtered
vector bundle with flat connection. Let
$(F^\bullet_i,\nabla_i)$ be the pullbacks of these
to $\bD$. Then, they descend to $X_\Gamma$
even when $\Gamma$ does not split as a product
of lattices $\Gamma_i\subset G_i$.
Let $V_i$ denote the $\C$-local system on $X_\Gamma$
of flat sections of $(F^0_i,\nabla_i)$.

\begin{definition} We define the 
{\it Hodge data of GAGA type}
$${\rm Hodge}_{X_\Gamma} = (F^\bullet_i, \nabla_i)_{1\leq i \leq k}$$ 
to be this $k$-tuple of
filtered flat vector bundles.
\end{definition}

\begin{remark} It is important to remark that 
the universal filtered flat vector bundle
$(F^\bullet,\nabla) = \textstyle\bigoplus_{i=1}^k 
(F^\bullet_i,\nabla_i)$
is not the same data of GAGA type as above! 
It may be impossible  to tell, a priori,
how $(F^\bullet,\nabla)$ splits, 
upon restriction to $Z\subset X_\Gamma$.
\end{remark}

\begin{remark}\label{above}
Let $Z\in \fD^{\Xi}$ be reduced and irreducible. 
Suppose $\wZ\to Z$ is a resolution of
singularities. Then $\wZ$ admits a $\Z$-PVHS by pulling
back $(V_\Z, F^\bullet,\nabla)$.
The pullback of ${\rm Hodge}_{X_\Gamma}=(F^\bullet_i,\nabla_i)_{1\leq i\leq k}$
constitutes the data of the collection
of simple factors of the $\C$-VHS.
Let $V$ be the local system of flat sections of
$\nabla_{\wZ}$.

The $\Z$-PVHS on $\wZ$, and thus, the period
map $\Phi\colon \wZ\to X_\Gamma$, is recoverable
from $(Z,{\rm Hodge}_Z)$ and
one critical missing piece of information:
the location of the integral lattice 
$V_{\Z,*}\hookrightarrow V_*$ in a fiber over
some base point $*\in \wZ$. This is the only data
which cannot be captured coherently on
$X_\Gamma$, and to which GAGA does not apply.
\end{remark}

Now, we leverage the fact that the lattice $V_{\Z,*}$
must be invariant under parallel transport.

\begin{proposition}\label{scale-factors}
Let $Z\in \fD^{\Xi}$ be irreducible and 
reduced, and suppose $\wZ\to Z$ is a resolution of 
singularities. Let $(V_\Z,F^\bullet)$ be the corresponding
pullback $\Z$-PVHS and let $*\in \wZ$ be a base point. 
Let $$\rho\colon \pi_1(\wZ,*)\to \GL(V_{\Z,*})$$
be the monodromy representation and let
$H=\prod_{i\in I} G_i\subset G$ be the collection of simple
factors in which
${\rm im}\,\rho$ is Zariski-dense. Fixing a frame of $V_{\Z,*}$,
the infinitesimal changes-of-frame giving rise to a lattice 
preserved by $\rho$ are contained in 
$$\prod_{i\in I}\C\times \prod_{i\notin I}\mathfrak{gl}(V_i).$$
\end{proposition}

\begin{proof} An infinitesimal change-of-frame $a\in \mathfrak{gl}(V_*)$ resulting in a new monodromy-invariant
lattice is exactly a matrix commuting with ${\rm im}(\rho)$,
and thus commuting with $\bfH(\R)$.
Since $V_i$ is an irreducible representation
of $(G_i)_\C$, Schur's lemma
implies that $a$ acts by
a scalar $\lambda_i$ on each summand $V_i\subset V$
for which $G_i\subset H$.
\end{proof}

\begin{definition} Given any analytic subspace
$Z\subset X_\Gamma$ we define
$\Gamma_Z$ as the image of $\pi_1(\wZ)\to \Gamma$ for some resolution of singularities $\wZ\to Z^{\rm red}$.\end{definition}

\begin{lemma}\label{normal2} Let $Z^\nu\to Z^{\rm red}$ be the
normalization. Then, $\Gamma_Z\subset \Gamma$ is the image
of $\pi_1(Z^\nu)$. It is also the image of $\pi_1(U)$
for any dense open subset $U\subset (Z^{\rm red})_{\rm sm}$. \end{lemma}

\begin{proof} Let $Z^\nu_{\rm sm}$ denote
the nonsingular locus. Then
$\pi_1(Z^\nu_{\rm sm})\onto \pi_1(Z^\nu)$
is surjective. The same property holds
for the inverse image of $Z_{\rm sm}^\nu$
or $U$ in any desingularization. Thus,
$\pi_1(\wZ)$, $\pi_1(Z^\nu_{\rm sm})$, $\pi_1(Z^\nu)$,
$\pi_1(U)$
all have the same image in $\Gamma=\pi_1(X_\Gamma)$. \end{proof}

\begin{proposition}\label{pi1} Let $Z\in \fD^{\Xi}$ be irreducible
and reduced. The group $\Gamma_Z$ only jumps in size,
in an open neighborhood of $Z\in \fD^{\Xi}$.
\end{proposition}

\begin{proof}
Let $(C,0)\to \fD^{\Xi}$
be an analytic arc, and consider the pullback
family $\fZ\to (C,0)$, with $\fZ_0=Z$. Let $\cW = \fZ^\nu$
be the normalization of the total space. The general fiber $\cW_t$ is normal, so
$\Gamma_{Z_t} = {\rm im}(\pi_1(\cW_t))$ by Lemma
\ref{normal2}.
This is the same group for all $t\in C\setminus 0$
if we assume (as we may) that $\cW$ is a fiber bundle
over $C\setminus 0$. There is a deformation-retraction
$\cW\to \cW_0$ to the central
fiber. Tracing an element of $\pi_1(\cW_t)$ through the
retraction, we get a free homotopy from any $\gamma_t\in \pi_1(\cW_t)$
to an element $\gamma_0\in \pi_1(\cW_0)$.

Conversely, we can lift any element of $\pi_1(\cW_0)$ 
to an element of $\pi_1(\cW_t)$: We have $\pi_1(\cW_0)=\pi_1(\cW) = 
\pi_1(\cW\setminus ((\cW_0)_{\rm sing}\cup \cW_{\rm sing}))$ because $\cW$ is
normal and $(\cW_0)_{\rm sing}\cup \cW_{\rm sing}$ 
has codimension $2$. Thus, any element of $\pi_1(\cW_0)=\pi_1(\cW)$ can
be represented by a loop in $\cW$ avoiding both $(\cW_0)_{\rm sing}$
and $\cW_{\rm sing}$. Then, 
this loop can be deformed off its intersection with $(\cW_0)_{\rm sm}$
as $(\cW_0)_{\rm sm}$ is a locally smooth divisor in $\cW_{\rm sm}$.
So we can represent the loop in $\cW\setminus \cW_0$.
Finally, $\pi_1(\cW\setminus \cW_0)$ is a $\Z$-extension of $\pi_1(\cW_t)$
because it is a fiber bundle over the punctured disk $C\setminus 0$.

Thus, $\Gamma_{Z_t}={\rm im}(\pi_1(\cW_0))$.
%Assume first that $Z$ is generically reduced.
Then the natural morphism
$\cW_0\to \fZ_0=Z$
is a finite birational
morphism because $Z$ is reduced.
Thus, it factors the normalization
$Z^\nu \to \cW_0 \to Z$
and so ${\rm im}(\pi_1(Z^\nu))=
\Gamma_{Z}\subset \Gamma_{Z_t}= {\rm im}(\pi_1(\cW_0)).$ Thus
$\Gamma_Z$ only jumps in size.
%(gets larger as a subset of $\Gamma$)
%in a neighborhood of $Z\in \fD^{\Xi}$. The first statement follows.
%Now, consider $Z$ generically non-reduced. Then
%$\cW_0\to Z^{\rm red}$ is finite, but may
%have degree larger than one---for instance,
%$\cZ$ might be a family of hyperelliptic curves
%collapsing to a doubled $\bP^1$ over $0\in C$,
%while $\cW$ is a smooth family of curves. But
%$\cW_0\to Z^{\rm red}$ is \'etale over
%a dense open subset $U\subset (Z^{\rm red})_{\rm sm}$.
%Thus a finite index subgroup of $\pi_1(U)$ lifts to %$\pi_1(\cW_0)$
%and so $\Gamma_{Z_t}$ contains a finite index subgroup of %$\Gamma_Z$.
%The second statement follows. 
\end{proof}

\begin{remark} The same statement holds, up to passing to
a finite index subgroup of $\Gamma_Z$, 
when $Z$ is generically non-reduced.
\end{remark}

\begin{theorem}\label{rigid-thm}
If $Z\subset X_\Gamma$ is irreducible, reduced,
and $\Gamma_Z$ is Zariski-dense in $G$, then any
irreducible component $\fC\subset (\fD^{\Xi})^{\rm red}$ containing
$Z$ is locally ${\rm Hodge}_{X_\Gamma}$-determined. In particular, $\fC$ is Moishezon by Corollary \ref{cvhs-embed}.\end{theorem}

\begin{proof} We must find an analytic open
set $U\subset \fC$ for which $$(Z_s, (F^\bullet_i,\nabla_i)_{1\leq i \leq k})\not\simeq (Z_t, (F_i^\bullet,\nabla_i)_{1\leq i\leq k})$$
for all $s\neq t\in U$. 
Choose $U$ to be a small neighborhood of $Z\in \fC$.
Since $Z$ is irreducible and reduced, 
we can assume that $Z_t$ is irreducible and reduced
for all $t\in U$. Applying Proposition \ref{pi1},
we ensure that all $Z_t\in U$ satisfy the property
that $\Gamma_{Z_t}$ is Zariski-dense in $G$.
It suffices to show there is no nonconstant
holomorphic
arc $C\to U$ for which the isomorphism type
of the tuple
$(Z_t,(F^\bullet_i,\nabla_i)_{1\leq i \leq k})$ 
is constant over $t\in C$.

Suppose for the sake
of contradiction that $C$ is such an arc. 
Choose a smooth base point $*\in Z_t=Z$ (observing
that $Z_t\simeq Z$ are isomorphic for all $t\in C$
by hypothesis).
Then by Proposition
\ref{scale-factors}, the only deformations of the lattice $V_{\Z,*}\subset V_*$ which remain invariant under $\nabla_{\wZ_t}=\bigoplus_{1\leq i\leq k} \nabla_i$ are those which differ by scaling each 
summand of $V=\bigoplus_{1\leq i\leq k} V_i$ by some $\lambda_i\in \C^*$.
But such scaling does not change
the period map, as the Hodge flag
$F^\bullet =\bigoplus_{1\leq i\leq k} F_i^\bullet$ is also
preserved by this scaling action.
But then the period maps
$\Phi_t\colon Z_t=Z\to X_\Gamma $ are all
the same, contradicting that $C$ parameterizes
a non-constant family of horizontal subspaces.
In other words, 
${\rm Hodge}_{X_\Gamma}$ is determinative on 
$U$.
\end{proof}

\begin{remark}\label{db} One could as easily have worked 
with Barlet spaces, since
the support morphism $[\cdot]\colon \fC\to \fB^{\Xi}$ will be
bimeromorphic onto its image, under the assumptions
of Theorem \ref{rigid-thm}. The disadvantage is that
the embedding into a compact, algebraic parameter space, as in
Example \ref{cvhs-embed}, is unclear for 
Barlet spaces. \end{remark}

\begin{theorem} Let $\cY\to \cS$ be a smooth projective
family over a quasiprojective variety $\cS$. Then
the non-abelian Hodge locus with $\Q$-anisotropic monodromy ${\rm NHL}_a(\cY/\cS,\GL_n)$ is algebraic.
\end{theorem}

 \begin{proof} 
Let $Y_s$ be a fiber. As we saw in \Cref{hodge-theory}, the data of a $\Z$-PVHS on $Y_s$ with generic Mumford-Tate group $\bfG\subset \GL_n$ and monodromy $\bfH$ is completely 
determined by the following data:
\begin{enumerate}
    \item a holomorphic, Griffiths transverse period map $\Phi_s:Y_s\rightarrow X_{\Gamma_H}$ whose monodromy image is Zariski-dense, and
    \item a point in $\bD_{H'}$ corresponding to a
    $\Q$-summand on which the $\Z$-PVHS is locally constant. 
\end{enumerate}

Thus, up to passing to a finite index subgroup
of fixed level, the monodromy representation
of such a $\Z$-PVHS has a reduction
of structure to the product
$\bfG = \bfH\times \bfH'$ where the corresponding
local system has trivial monodromy
on the summand associated to $\bfH'$.

Hence, possibly passing to a smaller value of $n$,
we can restrict our attention to the 
$(Y_s,\nabla_s)\in {\rm NHL}_a(\cY/\cS,\GL_n)$ which
underlie a $\Z$-PVHS $\bV$ with Zariski-dense monodromy
in the generic Mumford-Tate group.

By Corollary \ref{monodromy-finite},
only finitely many representations of $\pi_1(Y_s)$ with $\Q$-anisotropic monodromy
can appear in this manner. Thus, there is 
a finite list of compact Hodge manifolds $X_\Gamma$ which
receive all the period maps for such $(Y_s,\nabla_s)$.
So to prove the theorem,
we may restrict our attention to a single compact period
target $\Gamma\backslash \bD=X_\Gamma$.

It remains to show: The space of pairs 
$(Y_s,\Phi_s)$ of a fiber of $\cY\to \cS$, together with a 
Griffiths' transverse map $\Phi_s \colon Y_s\to X_\Gamma$
with Zariski-dense monodromy is an algebraic variety (and the maps into the 
relative de Rham and Dolbeault spaces are algebraic).
We first prove that each irreducible analytic
component of the space of pairs $(Y_s,\Phi_s)$ is algebraic, then we prove that the number of components is finite.
 
 Fix an irreducible analytic component $B\subset {\rm NHL}_a(\cY/\cS,\GL_n).$
 There is an analytic Zariski open subset $B^o\subset B$ on which ${\rm im}(\Phi_s)$,
 taken with its reduced scheme structure, form
 a flat family of closed analytic subspaces of $X_\Gamma$ 
 over $B^o$.
So there is an irreducible component $\fC\subset \fD^{\Xi}$
 for which ${\rm im}(\Phi_s)\in \fC$ for $(Y_s,\Phi_s)\in B^o$.
 
 Since $Y_s$ is smooth, the
 morphism $Y_s\to \Phi_s(Y_s)$ factors through
 the normalization $Y_s\to \Phi_s(Y_s)^\nu$. Thus,
 $\Gamma_{\rm im(\Phi_s)}$ contains the image
 of $\pi_1(Y_s)$ in $\Gamma$. 
Since we have restricted to the case
 where the monodromy is Zariski-dense,
 $\fC$ is Moishezon by Theorem \ref{rigid-thm}.
 
 Let $\fZ\to \fC$ be the universal family. For all $(Y_s,\Phi_s)\in B^o$,
 the period mapping $\Phi_s$ factors through the inclusion
 ${\rm im}(\Phi_s)\into \fZ$ as a fiber
 of the universal family. That is, we have a map
 $\Theta\colon \cY \times_{\cS} B^o\to \fZ$ for which
 $\Phi = \pi_{X_\Gamma}\circ \Theta$.
 
 The analytic deformations of $(Y_s,\Phi_s)$ in $B$
 are exactly the isomonodromic deformations of the local system
 $V_{\Z}$ on $Y_s$ to nearby fibers, which
 underlie a $\Z$-PVHS. But for $(Y_s,\Phi_s)\in B^o$,
 these are exactly the ways to deform the inclusion
 $\Theta_s\colon Y_s\hookrightarrow \fZ$ of fibers.
 Since $\cY \to \cS$ is algebraic
 and $\fZ\to \fC$ is Moishezon, the irreducible component of
 ${\rm Hom}_{\rm fiber}(\cY/\cS,\fZ)$, the
 space of morphisms from a fiber of $\cY$
 to a fiber of $\fZ$, which contains 
 $(Y_s,\Theta_s)\in B^o$, is Moishezon.

The inclusion into 
$M_{\rm dR}(\cY/\cS,\GL_n)$ is Moishezon because $\nabla_s$ is the pull back
along $\Theta_s$ of
the relative connection on $F^0$
on the universal
family over $\fZ\to \fC$. The relative
connection on $F^0$ is Moishezon, by GAGA.
Thus, $B^o$ and its closure $B$ are algebraic, as they are
Moishezon subsets of an algebraic
variety. The inclusion
into $M_{\rm Dol}(\cY/\cS,\GL_n)$ is Moishezon
by the same reasoning, applied
to the associated graded
of the universal Hodge flag
over $\fZ\to\fC$, equipped with its Higgs field.\vspace{2pt}

Finally, it remains to prove that (1) only finitely
many irreducible components $\fC$ of the horizontal
Douady space appear, and (2) for each one that
appears, the number of irreducible components
of the space ${\rm Hom}_{\rm fiber}(\cY,\fZ)$
is finite.

Let $F^\bullet$ be 
the Hodge filtration on $Y_s$ coming from a period map 
$\Phi_s\colon Y_s\to X_\Gamma$ and let $A\to \cY$ be an 
ample line bundle on the universal family. 
Then by Simpson
\cite[Lemma 3.3]{simpson-moduli-1},
the vector bundles $F^p$ enjoy the following version
of the Arakelov inequality: 
If $m_s$ is an integer for which $T_{Y_s}(m_sA)$ 
is globally generated, then 
$\mu_A(F^{p+1})\leq \mu_A(F^p)+mn$. 
Here $\mu_A$ is the slope with respect to $A$. 
Note that $\mu_A(F^0)=0$
because $F^0$ has a flat structure.
We may choose an $m_s=m$ uniformly over all of $\cS$.
We conclude that the slopes $\mu_A(F^p)$ are bounded,
in a way depending only on $\cY\to \cS$.
In turn, $A^{d-1}\cdot \det(F^p)$
is bounded for all $p$, and so there is an a priori
bound on $A^{d-1}\cdot L$, where $L$ is the Griffiths bundle. It follows that $A^{d-r}\cdot L^r$ is 
bounded for any $r$.

This bounds the Griffiths volume of the image
$\Phi_s(Y_s)$ of any period map, and so by Theorem \ref{properness}, only finitely many components of the horizontal Barlet space $\fB^{\Xi}$ of $X_\Gamma$ occur as 
the support of period images from $Y_s$. The same
finiteness holds for relevant components $\fC$ of the horizontal
Douady space, as we are taking period images with their
reduced scheme structure, see Remark \ref{db}.

Finally, the bounds on $A^{d-r}\cdot L^r$ also
bound the volume of the graph $\Gamma(\Theta_s)$ of a morphism 
$(Y_s,\Theta_s)\in {\rm Hom}_{\rm fiber}(\cY,\fZ)$, viewed as 
a subvariety of $\cY\times \fZ$. We conclude that there
must be only finitely many components of 
${\rm Hom}_{\rm fiber}(\cY,\fZ)$.
 \end{proof}

\begin{remark}
    The algebraicity result also holds for $\bfG$-bundles, for any algebraic subgroup $\bfG\subset \mathrm{GL}_n$. 
\end{remark}

It is straightforward to construct $\Z$-PVHS with $\Q$-anisotropic monodromy and which are of Shimura type, by taking subvarieties of compact Shimura varieties and it would be interesting to construct examples which are not of Shimura type. Notice also that in the Shimura case, the proof of algebraicity of the irreducible components of the non-abelian Hodge locus are easier, as the arithmetic quotients of the period domains involved are algebraic varieties. But the finiteness of monodromy representations was only known for Shimura varieties of abelian type by \cite{deligne87}. As a corollary of our work, we obtain the following:

\begin{corollary}
Let $\bfG$ be a Shimura group of exceptional type. There are only finitely many representations $\rho:\Pi_g\rightarrow \bfG(\R)$ which underlie a $\Z$-PVHS with $\Q$-anisotropic monodromy, up to the action of the mapping class group, and in $M_{\rm dR}(\cC_g/\cM_g, \GL_n)$, the corresponding
flat bundles form an algebraic subvariety. 
\end{corollary}

\bibliographystyle{alpha}
\bibliography{bibliographie}

\end{document}